\documentclass[11pt]{article}

\usepackage{amsfonts,amssymb,amsmath,bm}
\usepackage{a4wide}\voffset=-5mm\hoffset=1mm
\def\N{\mathbb{N}}
\def\Z{\mathbb{Z}}
\def\Q{\mathbb{Q}}
\def\R{\mathbb{R}}

\def\C{\mathbb{C}}

\newcommand{\cB}{\mathcal{B}}
\newcommand{\cC}{\mathcal{C}}
\newcommand{\cD}{\mathcal{D}}

\newcommand{\cF}{\mathcal{F}}

\newcommand{\cI}{\mathcal{I}}
\newcommand{\cJ}{\mathcal{J}}
\newcommand{\cK}{\mathcal{K}}

\newcommand{\cM}{\mathcal{M}}

\newcommand{\cR}{\mathcal{R}}

\newcommand{\cV}{\mathcal{V}}
\newcommand{\cW}{\mathcal{W}}

\newcommand{\bb}{b}    
\newcommand{\pp}{}
\newcommand{\ppp}{}

\newcommand{\cKI}{\cK((\cI_q)_{q\ge0})}

\newtheorem{theorem}{Theorem}
\newtheorem{Blichfeldt}{Theorem (Blichfeldt \cite{Blichfeldt-1921})}

\newtheorem{Minkowski}{Minkowski's Convex Body Theorem (see \cite[Theorem\,2B]{Schmidt-1980})}

\newtheorem{lemma}{Lemma}
\newtheorem{corollary}{Corollary}
\newtheorem{proposition}{Proposition}

\newenvironment{proof}{\noindent\textit{{Proof}}.}{\hfill\raisebox{-1ex}{$\boxtimes$}}
\newcommand{\qed}{\hfill\raisebox{-1ex}{$\boxtimes$}}

\newcommand{\vv}[1]{{\bf{#1}}}

\newcommand{\ve}{\varepsilon}
\newcommand{\rank}{{\rm rank\,}}

\newcommand{\wcJ}{\widehat\cJ}
\newcommand{\wcI}{\widehat\cI}
\newcommand{\qqand}{\qquad\text{and}\qquad}
\newcommand{\qand}{\quad\text{and}\quad}
\newcommand{\spn}{\operatorname{span}}
\newcommand{\Bad}{\mathbf{Bad}}
\newcommand{\Badn}{\mathbf{Bad}(n)}
\newcommand{\rr}{\vv r}

\newcommand{\Badr}{\mathbf{Bad}(\rr)}
\newcommand{\vol}{\mathrm{vol}}

\newcommand{\gtr}{g^t_{\rr,\bb}}
\newcommand{\gt}{g^t}
\newcommand{\Gx}{G_x}

\newcommand{\Gky}{G(\kappa;\vv y)}
\begin{document}

\large

\title{\bf Badly approximable points on manifolds}

\author{\sc Victor Beresnevich\footnote{Supported by EPSRC grant EP/J018260/1\newline
\noindent Victor Beresnevich: University of York, Heslington, York, YO10 5DD, UK\newline
E-mail : {\tt victor.beresnevich@york.ac.uk}} {~(York)}
}

\date{}
\maketitle

\vspace*{-5ex}

\begin{abstract}
This paper is motivated by two problems in the theory of Diophantine approximation, namely, Davenport's problem regarding badly approximable points on submanifolds of a Euclidean space and Schmidt's problem regarding the intersections of the sets of weighted badly approximable points. The problems have been recently settled in dimension two but remain open in higher dimensions. In this paper we develop new techniques that allow us to tackle them in full generality. The techniques rest on lattice points counting and a powerful quantitative result of Bernik, Kleinbock and Margulis. The main theorem of this paper implies that any finite intersection of the sets of weighted badly approximable points on any analytic nondegenerate submanifold of $\R^n$ has full dimension. One of the consequences of this result is the existence of transcendental real numbers badly approximable by algebraic numbers of any bounded degree.
\end{abstract}

{\small \noindent\emph{Keywords}: Diophantine approximation on manifolds, badly approximable points, Schmidt's conjecture, approximation by algebraic numbers}

\section{Introduction}

The notion of badly approximable numbers, as much of the classical and modern theory of Diophantine approximation, is underpinned by Dirichlet's fundamental result. It states that for every $\alpha\in\R$ and any $Q>1$ there exists $q\in\N$ and $p\in\Z$ such that $|q\alpha-p|<Q^{-1}$ and $q\le Q$. In particular, it implies that
for every real irrational number $\alpha$ the inequality
$$
\left|\alpha -\frac pq\right|< \frac1{q^{2}}
$$
holds for infinitely many rational numbers $p/q$ written as reduced fractions of integers $p$ and $q$.
A real number $\alpha$ is then called \emph{badly approximable} if there exists a constant $c=c(\alpha)>0$ such that
\begin{equation}\label{dir}
\left|\alpha -\frac pq\right|\ge \frac{c}{q^2}
\end{equation}
for all $(q,p)\in\N\times\Z$. In what follows, the set of badly approximable real numbers will be denoted by $\Bad$.

It is well known that a real irrational number $\alpha$ is badly approximable if and only if the partial quotients of its continued fraction expansion are uniformly bounded. For instance, any real quadratic irrational number is in $\Bad$, since its continued fraction expansion is eventually periodic\footnote{It is not known whether there are any real algebraic numbers of degree $\ge3$ that are badly approximable.}. Using continued fractions one can easily produce continuum many examples of badly approximable real numbers. Beyond the cardinality, Jarn\'ik \cite{Jarnik-28} established that $ \dim \Bad$ (the Hausdorff dimension of $\Bad$) is $1$. However,
the Lebesgue  measure of $\Bad$ is known to be zero. This is a trivial consequence of the divergence case of Khintchine's theorem \cite{Khintchine-1924}, and can also be relatively easily proved using the Lebesgue density theorem, see \cite{Cass} or \cite[Corollary~2]{BV-zoo}.

\subsection{Higher dimensions: Schmidt's conjecture}\label{Schmidt}

Higher dimensions offer various ways of generalising the notion of badly approximable numbers. For now, we restrict ourselves to considering simultaneous Diophantine approximations by rationals. The point $\vv y=(y_1,\dots,y_n)\in\R^n$ is called \emph{badly approximable}\/ if there exists a constant $c=c(\vv y)>0$ such that
\begin{equation}\label{vb200}
\max_{1\le i\le n}\|qy_i\|\ge cq^{-1/n}
\end{equation}
for all $q\in\N$, where $\|x\|$ denotes the distance of $x$ from the nearest integer. The quantities $\|qy_i\|$ are equal to $|qy_i-p_i|$ for some $p_i\in\Z$ and thus give rise to `approximating' rationals $p_1/q,\dots,p_n/q$. Once again, the notion of badly approximable points is underpinned by Dirichlet's theorem, this time for $\R^n$, which implies that the inequality $\max_{1\le i\le n}\|qy_i\|<q^{-1/n}$ holds for infinitely many $q\in\N$.
The set of badly approximable points in $\R^n$ will be denoted by $\Badn$. Observe that $\Bad(1)=\Bad$.

The first examples of badly approximable points in $\R^n$ were given by Perron \cite{Perron-21:MR1512000} who used an algebraic construction and produced infinitely yet countably many elements of $\Badn$. For instance, $(\alpha,\dots,\alpha^n)\in\Bad(n)$ whenever $\alpha$ is a real algebraic number of degree $n+1$. However, it was not until 1954 when first Davenport \cite{Davenport-54:MR0064083} for $n=2$ and then Cassels \cite{Cassels-55:MR0075240} for $n\ge2$ showed that $\Badn$ was uncountable. The fact that $\Badn$ has full Hausdorff dimension was proved by Schmidt \cite{Schmidt-66:MR0195595} who introduced powerful ideas based on a specific type of games. The dimension result for $\Badn$ comes about as a consequence of the fact that $\Badn$ is winning for Schmidt's game. Furthermore, Schmidt proved that affine transformations of $\Badn$ are winning and that the collection of winning sets in $\R^n$ is closed under countable intersections.

In his 1983 paper \cite{SC} Schmidt formulated a conjecture that later became the catalysis for some remarkable developments. Schmidt's conjecture rests on the modified notion of badly approximable points in which approximations in each coordinate are given some weights, say $r_1,\dots,r_n$. In short, he conjectured that there exist points in $\R^2$ that are simultaneously badly approximable with respect to two different collections of weights. The weights of approximation are required to satisfy the following conditions:
\begin{equation}\label{weights}
r_1+\ldots+r_n=1\qquad\text{and}\qquad r_i\ge0\text{ ~~for all~ }i=1,\dots,n\,.
\end{equation}
Throughout this paper the set of all $n$-tuples $\rr=(r_1,\dots,r_n)$ subject to \eqref{weights} will be denoted by $\cR_n$. Formally, given $\rr\in\cR_n$, the point $\vv y=(y_1,\dots,y_n)\in\R^n$ will be called \emph{$\rr$-badly approximable} if there exists $c=c(\vv y)>0$ such that
\begin{equation}\label{bad2}
\max_{1\le i\le n}\|qy_i\|^{1/r_i}\ge cq^{-1}
\end{equation}
for all $q\in\N$. Here, by definition, $\|qy_i\|^{1/0}=0$. Again, a version of Dirichlet's theorem tells us that when $c=1$ inequality \eqref{bad2} fails infinitely often.

The set of $\rr$-badly approximable points in $\R^n$ will be denoted by $\Badr$. As is readily seen, the classical set of badly approximable points $\Badn$ is simply $\Bad(\tfrac1n,\dots,\tfrac1n)$.
Using this notation we can now specify the following concrete statement conjectured by Schmidt:
$$
\Bad(\tfrac13,\tfrac23)\cap\Bad(\tfrac23,\tfrac13)\neq\emptyset.
$$

It is worth mentioning that the sets $\Badr$ have been studied at length in all dimensions and for arbitrary collections of weights, see \cite{Davenport-64:MR0166154, Pollington-Velani-02:MR1911218, Kleinbock-Weiss-05:MR2191212, Kristensen-Thorn-Velani-06:MR2231044, Kleinbock-Weiss-10:MR2581371}. Partly the interest was fueled by natural links with homogeneous dynamics and Littlewood's conjecture in multiplicative Diophantine approximation, another long standing problem -- see \cite{Badziahin-Pollington-Velani-Schmidt} for further details. Schmidt's conjecture withstood attacks for nearly 30 years. However, the recent progress has been dramatic.

In 2011 Badziahin, Pollington \& Velani \cite{Badziahin-Pollington-Velani-Schmidt} made a breakthrough by proving that for any sequence $\rr_k=(i_k,j_k)\in\cR_2$ such that
\begin{equation}\label{eee02}
\liminf_{k\to\infty}\min\{i_k,j_k\}>0
\end{equation}
and any vertical line $L_\theta=\{(\theta,y):y\in\R\}\subset\R^2$ with $\theta\in\Bad$ one has that
\begin{equation}\label{dbsv}
\textstyle\dim\bigcap_k\Bad(\rr_k)\cap L_\theta=1.
\end{equation}
This readily gives that $\dim\bigcap_k\Bad(\rr_k)=2$ and proves Schmidt's conjecture in a much stronger sense. Shortly thereafter, An \cite{An-12} proves that for any $\rr\in\cR_2$ and any $\theta\in\Bad$ the set $\Bad(\rr)\cap L_\theta$ is winning for a Schmidt game in $L_\theta$. This immediately leads him to removing condition (\ref{eee02}) from the theorem of Badziahin, Pollington \& Velani, since the collection of Schmidt's winning sets is closed under arbitrary countable intersections. In a related paper An \cite{An-12-2} establishes that $\Bad(i,j)$ is winning for the $2$-dimensional Schmidt game, thus giving another proof of Schmidt's conjecture. Generalising the techniques of \cite{Badziahin-Pollington-Velani-Schmidt} in yet another direction Nesharim \cite{Nes-12}, independently from An, proves that the set in the left hand side of (\ref{dbsv}) intersected with naturally occurring fractals embedded in $L_\theta$ is uncountable for any sequence $(\rr_k)_{k\in\N}$. Subsequently, Nesharim jointly with Weiss establishes the winning property of these intersections -- see Appendix B in \cite{Nes-12}.

As already mentioned, the sets $\Badr$ and even their restrictions to naturally occurring fractals have been investigated in higher dimensions, see \cite{Kleinbock-Weiss-05:MR2191212, Kristensen-Thorn-Velani-06:MR2231044, Fishman-09:MR2520102, Kleinbock-Weiss-10:MR2581371}. In particular, the sets $\Badr$ were shown to have full Hausdorff dimension for any $\rr\in\cR_n$. However, the theory of their mutual intersections is a different story. In an apparent attempt to prove Schmidt's conjecture, Kleinbock and Weiss \cite{Kleinbock-Weiss-10:MR2581371} introduced a modified version of Schmidt's games. As they have shown, winning sets for the same modified Schmidt game inherit the properties of classical winning sets. Namely, they have full Hausdorff dimension and their countable intersections are winning with respect to the same game. Also Kleinbock and Weiss have proved that $\Badr$ is winning for a relevant modified Schmidt game. However, it was not possible to prove that the intersection $\Bad(\rr_1)\cap\Bad(\rr_2)$ was a winning set for some modified Schmidt game as, with very few exceptions, the corresponding modified Schmidt games were not `compatible'. As a result the following key problem that generalises Schmidt's original conjecture has remained open in dimensions $n\ge3$:

\medskip

\noindent\textbf{Problem 1:} {\em Let $n\in\N$. Prove that for any finite or countable subset $W$ of $\cR_n$ one has that
\begin{equation}\label{vb1}
\dim\bigcap_{\rr\in W}\Badr=n\,.
\end{equation}}

The main result of this paper implies \eqref{vb1} in arbitrary dimensions $n$ and for arbitrary countable subsets $W$ of weights satisfying a condition similar to \eqref{eee02}. For instance, the result is applicable to arbitrary finite collections of weights $W$. The proof will be given by restricting the sets of interest to a suitable family of curves in $\R^n$. Interestingly, this approach, which was innovated in \cite{Badziahin-Pollington-Velani-Schmidt} in the case $n=2$, turns out to face another intricate problem that was first communicated by Davenport.

\subsection{$\Badr$ on manifolds and Davenport's problem}

In 1964 Davenport \cite{Davenport-64:MR0166154} established that, given a finite collection $\vv f_i:\R^m\to\R^{n_i}$ $(1\le i\le N)$ of $C^1$ maps, if for some $\vv x_0\in\R^m$ and every $i=1,\dots,N$ the Jacobian of $\vv f_i$ at $\vv x_0$ has rank $n_i$, then the set\\[-3ex]
$$
\textstyle\bigcap\limits_{i=1}^N\vv f_i^{-1}(\Bad(n_i))
$$
has the power of continuum. For instance, taking $f_1(x,y)=x$, $f_2(x,y)=y$ and $\vv f_3(x,y)=(x,y)$ shows that $\Bad(1,0)\cap\Bad(\tfrac12,\tfrac12)\cap\Bad(0,1)$ has the power of continuum. Another natural example obtained by taking $f_i(x)=x^i$ for $i=1,\dots,k$ shows that there are continuum many $\alpha\in\R$ such that $\alpha,\alpha^2,\dots,\alpha^k$ are all in $\Bad$.

Clearly, the Jacobian condition above implies that $m\ge n_i$ for every $i$. Commenting on this, Davenport writes \cite[p.\,52]{Davenport-64:MR0166154} ``\textsl{Problems of a much more difficult character arise when the number of independent parameters is less than the dimension of simultaneous approximation. I do not know whether there is a set of $\alpha$ with the cardinal of the continuum such that the pair $(\alpha,\alpha^2)$ is badly approximable for simultaneous approximation.}'' Essentially, if $m<n_i$ then $\vv f_i(\vv x)$ lies on a submanifold of $\R^{n_i}$. Hence, Davenport's problem boils down to investigating badly approximable points restricted to submanifolds of Euclidean spaces.

In the theory of Diophantine approximation on manifolds, see for instance, \cite{BD99, Beresnevich-02:MR1905790, Beresnevich-12, Kleinbock-03:MR1982150, Kleinbock-Margulis-98:MR1652916}, there are already well established classes of manifolds of interest. These include non-degenerate manifolds and affine subspaces and should likely be of primary interest when resolving Davenport's problem.

It is worth pointing out that the result of Perron \cite{Perron-21:MR1512000} mentioned in \S\ref{Schmidt} implies the existence of algebraic badly approximable points on the Veronese curves $\cV_n=\{(x,\dots,x^n):x\in\R\}$. However, there are only countably many of them. Khintchine \cite{Khintchine-1925} proved that $\Badn\cap\cV_n$ had zero 1-dimensional Lebesgue measure. Baker \cite{Baker-1976} generalised this to arbitrary $C^1$ submanifold of $\R^n$. Apparently, $\Badn$ can be relatively easily replaced with $\Badr$ in Baker's result, though, to the best of author's knowledge, this has never been formally addressed. To make a long story short, until recently there has been no success in relation to Davenport's problem even for planar curves, let alone manifolds in higher dimension.
The aforementioned work of Badziahin, Pollington and Velani \cite{Badziahin-Pollington-Velani-Schmidt} was the first step forward. Very recently, assuming (\ref{eee02}), Badziahin and Velani \cite{Badziahin-Velani-Dav} have proved (\ref{dbsv}) with $L_\theta$ replaced by any $C^2$ planar curve which is not a straight line. In particular, this shows that there exist uncountably many real numbers $\alpha$ such that $(\alpha,\alpha^2)$ is in $\Bad(2)$. Also they have dealt with a family of lines in $\R^2$ satisfying a natural Diophantine condition. The most recent results established in \cite{ABV} by An, Velani and the author of this paper remove condition (\ref{eee02}) from the findings of \cite{Badziahin-Velani-Dav} and at the same time settle Davenport's problem for a larger class of lines in $\R^2$ defined by a near optimal condition. As a result, the following general version of Davenport's problem is essentially settled in the case $n=2$:

\medskip

\noindent\textbf{Problem~2:} \emph{Let $n,m\in\N$, $B$ be a ball in $\R^m$, $W$ be a finite or countable subset of $\cR_n$ and $\cF_n(B)$ be a finite or countable collection of maps $\vv f:B\to\R^n$. Determine sufficient (and possibly necessary) conditions on $W$ and/or $\cF_n(B)$ so that
\begin{equation}\label{vb1++}
\dim\bigcap_{\vv f\in\cF_n(B)}\ \bigcap_{\rr\in W}\vv f^{-1}(\Bad(\rr))=m\,.
\end{equation}
}

Despite the success in resolving Problem~2 for planar curves, no progress has been made on Davenport's problem for $n\ge3$. The results of this paper imply \eqref{vb1++} in arbitrary dimensions $n$ and for arbitrary countable subsets $W$ of weights satisfying a condition similar to \eqref{eee02} and arbitrary finite collection $\cF_n(B)$ of analytic non-degenerate maps. The proof introduces new ideas based on lattice points counting and a powerful quantitative result of Bernik, Kleinbock and Margulis. Indeed, the arguments presented should be of independent interest even for $n=2$.

\section{Main results and corollaries}

In what follows, an analytic map $\vv f:B\to\R^n$ defined on a ball $B\subset\R^m$ will be called \emph{nondegenerate} if the functions $1,f_1,\dots,f_n$ are linearly independent over $\R$. The more general notion of nondegeneracy that does not require analyticity can be found in \cite{Kleinbock-Margulis-98:MR1652916}.
Given an integer $n\ge2$, $\cF_n(B)$ will denote a family of maps $\vv f:B\to\R^n$ with a common domain $B$. To avoid ambiguity, let us agree from the beginning that all the intervals and balls mentioned in this paper are of positive and finite diameter. Recall that $\cR_n$ denotes the collection of weights of approximation and is defined by \eqref{weights}. Given $\rr=(r_1,\dots,r_n)\in\cR_n$, let
\begin{equation}\label{eee03}
    \tau(\rr)\stackrel{\rm def}{=} \min\{r_i:r_i\neq0\}\,,
\end{equation}
that is $\tau(\rr)$ is the smallest strictly positive weight within $\rr$. The following result regarding Problem~2 represents the main finding of this paper.

\begin{theorem}\label{t1}
Let $m,n\in\N$, $1\le m\le n$, $B$ be an open ball in $\R^m$ and $\cF_n(B)$ be a finite family of analytic nondegenerate maps. Let $W$ be a finite or countable subset of $\cR_n$ such that
\begin{equation}\label{eq+01}
    \inf\{\tau(\rr):\vv r\in W\}>0\,.
\end{equation}
Then \eqref{vb1++} is satisfied.
\end{theorem}

\medskip

\noindent Condition (\ref{eq+01}) matches (\ref{eee02}) and is satisfied whenever $W$ is finite. Now we consider the following basic corollary regarding badly approximable points on manifolds.

\begin{corollary}\label{cor1}
Let $\cM$ be a manifold immersed into $\R^n$ by an analytic nondegenerate map. Let $W\subset\cR_n$ be a finite or countable set of weights. Assume that \eqref{eq+01} is satisfied. Then
$
\textstyle \dim\bigcap_{\vv r\in W}\Bad(\rr)\cap\cM=\dim\cM\,.
$
In particular, for any finite collection $\rr_1,\dots,\rr_N\in\cR_n$ we have that
$$
\textstyle\dim\bigcap_{k=1}^N\Bad(\rr_k)\cap\cM=\dim\cM.
$$
\end{corollary}

Note that the corollary is applicable to $\cM=\R^n$, which is clearly analytic and nondegenerate.
In this case Corollary~\ref{cor1} establishes an analogue of Schmidt's conjecture in arbitrary dimensions $n\ge2$ by settling Problem~1 subject to condition \eqref{eq+01}.

\subsection{Reduction to curves}\label{RC}

When $m=1$ the nondegeneracy of an analytic map $\vv f=(f_1,\dots,f_n)$ is equivalent to the Wronskian of $f'_1,\dots,f'_n$ being not identically zero. More generally, the map $\vv f$ (not necessarily analytic) defined on an interval $I\subset\R$ will be called \emph{nondegenerate at $x_0\in I$}\/ if $\vv f$ is $C^n$ on a neighborhood of $x_0$ and the Wronskian of $f'_1,\dots,f'_n$ does not vanish at $x_0$.  This definition of nondegeneracy at a single point is adopted within the following more general result for curves. Note that if $\vv f$ is nondegenerate at least at one point, then the functions $1,f_1,\dots,f_n$ are linearly independent over $\R$.

\begin{theorem}\label{t2}
Let $n\in\N$, $n\ge 2$, $I\subset\R$ be an open interval and $\cF_n(I)$ be a finite family of maps defined on $I$ nondegenerate at the same point $x_0\in I$. Let $W$ be a finite or countable subset of $\cR_n$ satisfying \eqref{eq+01}. Then
\begin{equation}\label{thm1}
\dim\bigcap\limits_{\vv f\in\cF_n(I)}\ \bigcap\limits_{\vv r\in W}\vv f^{-1}(\Bad(\rr))=1\,.
\end{equation}
\end{theorem}

Our immediate goal is to show that Theorem~\ref{t1} is a consequence of Theorem~\ref{t2}. In metric Diophantine approximation the idea of reducing the case of manifolds to curves is not new. For instance, Badziahin, Pollington \& Velani use fibering of $\R^2$ into vertical lines in their proof of Schmidt's conjecture \cite{Badziahin-Pollington-Velani-Schmidt}. Underpinning our reduction of Theorem~\ref{t1} to Theorem~\ref{t2} is the following version of Marstrand's slicing lemma, see \cite[Corollary~7.12]{Falconer-03:MR2118797} or \cite[Theorem~10.11]{Mat}.

\bigskip

\noindent\textbf{Marstrand's slicing lemma:} {\em Let $m>1$ and $S$ be a subset of $\R^m$. Let $s>0$ and let
$U$ be a subset of $\R^{m-1}$ such that $\dim\{(t,u_2,\dots,u_m)\in S\}\ge s$ for each $(u_2,\dots,u_m)\in U$.
Then
$$
\dim S\ge\dim U+s\,.
$$
}

We will also need the following formal statement which is a slightly modified extract from Sprind\v zuk's survey \cite[pp.\,9-10]{Sprindzuk-1980-Achievements}.

\bigskip

\noindent\textbf{The Fibering Lemma\,:}\label{SFL}
{\em  Let $f_0,\dots,f_n$ be analytic functions in $m$ real variables defined on an open neighborhood of $\vv0$. Assume that $f_0,\dots,f_n$ are linearly independent over $\R$. Then there is a sufficiently large integer $d_0>1$ such that for every $d> d_0$ and every $\vv u=(u_1,u_2,\dots,u_m)\in \R^{m}$ with $u_1\cdots u_m\neq0$ the following functions of one real variable
  $$
  \phi_{\vv u,i}:E_{\vv u}\to\R\quad (0\le i\le n)
  $$
  given by
  $$
  \phi_{\vv u,i}(t)\stackrel{\rm def}{=} f_i(u_1t^{1+d^m},u_2t^{d+d^m},\dots,u_mt^{d^{m-1}+d^m})\,,
  $$
where $E_{\vv u}\subset\R$ is a neighbourhood of\/ $0$, are linearly independent over $\R$.}

\bigskip

Although the proof of the Fibering Lemma mostly follows the argument of \cite[pp.\,9-10]{Sprindzuk-1980-Achievements}, for completeness full details are given in \ref{C}. Note that Sprind\v zuk's version of fibering involves the parametrisation $\widetilde\phi_{\vv u,i}(t)=f_i(u_1t,u_2t^{d},\dots,u_mt^{d^{m-1}})$.

\bigskip

\noindent\textit{Proof of Theorem~\ref{t1} modulo Theorem~\ref{t2}.}
Let $\cF_n(B)$ be as in Theorem~\ref{t1} and let $\vv f=(f_1,\dots,f_n)\in\cF_n(B)$. Without loss of generality we will assume that $B$ is centred at $\vv0$. Also assume that $m\ge 2$ as otherwise there is nothing to prove. Let $u_1=1$, $t_0>0$ and $\delta_2, \dots, \delta_m>0$ be sufficiently small numbers such that
$$
(t^{1+d^m},u_2t^{d+d^m},u_3t^{d^2+d^m},\dots,u_mt^{d^{m-1}+d^m})\in B
$$
whenever
\begin{equation}\label{sl}
\tfrac12t_0<t<t_0,\qquad \tfrac12\delta_i<u_i<\delta_i\quad (2\le i\le m)\,.
\end{equation}
The existence of $t_0,\delta_2,\dots,\delta_m$ is guaranteed by the fact that $\vv0$ is an interior point of $B$. Let $U$ be the set of $\vv u=(u_2,\dots,u_m)$ satisfying the right hand side inequalities of \eqref{sl} and
$D$ be the set of $(t,u_2,\dots,u_m)$ satisfying \eqref{sl}.

By the nondegeneracy of $\vv f$, the functions $1,f_1,\dots,f_n$ are linearly independent. Since they are also analytic, by the Fibering Lemma, there exists $d_0(\vv f)>0$ such that for every $d>d_0(\vv f)$ and every $\vv u\in U$ the coordinate functions of the map
\begin{equation}\label{fu}
\vv f_{\vv u}(t)=\vv f\big(t^{1+d^m},u_2t^{d+d^m},u_3t^{d^2+d^m},\dots,u_mt^{d^{m-1}+d^m}\big)
\end{equation}
defined on the interval $I=(\tfrac12t_0,t_0)$ together with $1$ are linearly independent over $\R$.
Since $\cF_n(B)$ is finite,
$$
d_0 {=} \max\{d_0(\vv f):\vv f\in\cF_n(B)\}
$$
is well defined. Let $d>d_0$. Then for every $\vv f\in\cF_n(B)$ and every
$\vv u\in U$ the coordinate functions of the map \eqref{fu} together with $1$ are linearly independent over $\R$. By the well known criterion of linear independence, their Wronskian is not identically zero. Hence, the Wronskian of $\vv f'_{\vv u}= \frac{d}{dt}\vv f_{\vv u}$ is not identically zero. As an analytic function, it has isolated zeros. Hence, for a fixed $\vv u$, there are at most countably many points in $I$ where the Wronskian of $\vv f'_{\vv u}$ vanishes for some $\vv f\in\cF_n(B)$. Hence, there exists a point $x_0\in I$, which may depend on $\vv u$, such that for every $\vv f\in\cF_n(B)$ the Wronskian of $\vv f'_{\vv u}$ is not zero, that is $\vv f_{\vv u}$ is non-degenerate at $x_0$. Thus, Theorem~\ref{t2} is applicable and we conclude that the following subset of $I$
$$
S_{\vv u} = \bigcap_{\vv f\in\cF_n(B)}\ \bigcap_{\vv r\in W}\vv f_{\vv u}^{-1}(\Bad(\rr))
$$
has Hausdorff dimension $1$. Here, by definition, $\vv f_{\vv u}^{-1}(\Bad(\rr))$ is the set of $t\in I=(\tfrac12t_0,t_0)$ such that $\vv f_{\vv u}(t)\in\ \Badr$. Then, by Marstrand's slicing lemma, the set
$$
S=\big\{(t,u_2,\dots,u_m):t\in S_{\vv u},\ \vv u\in U\big\}\subset D
$$
has Hausdorff dimension $\ge \dim U+1=m$.
Let $S'\subset B$ be the image of $S$ under the map
\begin{equation}\label{map2}
(t,u_2,\dots,u_m)\mapsto(x_1,\dots,x_m)\stackrel{\rm def}{=} (t^{1+d^m},u_2t^{d+d^m},u_3t^{d^2+d^m},\dots,u_mt^{d^{m-1+d^m}})\,.
\end{equation}
Then, in view of the definitions of $S$, $S_{\vv u}$ and $\vv f_{\vv u}$, we have that
\begin{equation}\label{incl0}
S'\subset \bigcap_{\vv f\in\cF_n(B)}\ \bigcap_{\vv r\in W}\vv f^{-1}(\Bad(\rr))\,.
\end{equation}
Further, note that \eqref{map2} maps $D$ into $B$ injectively and is bi-Lipschitz on $D$, since the map itself and its inverse (defined on the image of $D$) have continuous bounded derivatives. It is well known that bi-Lipschitz maps preserves Hausdorff dimension, see for example \cite[Corollary~2.4]{Falconer-03:MR2118797}. Therefore,
$\dim S'=\dim S\ge m$. By \eqref{incl0}, and the fact that any subset of $\R^m$ is of dimension $\le m$, we obtain \eqref{vb1++} and thus complete the proof of Theorem~\ref{t1} modulo Theorem~\ref{t2}.\\[-1ex]
\hspace*{\fill}$\boxtimes$

\subsection{The dual form of approximation}

So far we have been dealing with simultaneous rational approximations. Here we introduce the dual definition of badly approximable points -- see part (iii) of Lemma~\ref{dual} below. This has two purposes. Firstly, it is the dual form that will be used in the proof of the results. Secondly, the dual form provides a natural environment for considering Diophantine approximation by algebraic numbers and will allow us to deduce further corollaries of our main results.

\begin{lemma}[Equivalent definitions of $\Badr$]\label{dual}
Let $\rr=(r_1,\dots,r_n)\in\cR_n$ and $\vv y=(y_1,\dots,y_n)\in\R^n$. Then the following three statements are equivalent:
\begin{itemize}
  \item[\rm(i)] $\vv y\in\Badr.$

  \item[\rm(ii)] There exists $c>0$ such that for any $Q\ge1$ the only integer solution $(q,p_1,\dots,p_n)$ to the system
\begin{equation}\label{eee10+}
|q|< Q, \qquad |qy_i-p_i| < \left(c\,Q^{-1}\right)^{r_i}\quad (1\le i\le n)
\end{equation}
is $q=p_1=\dots=p_n=0$.

  \item[\rm(iii)]
There exists $c>0$ such that for any $H\ge1$ the only integer solution $(a_0,a_1,\dots,a_n)$ to the system
\begin{equation}\label{eee10}
|a_0+a_1y_1+\dots+a_ny_n|< c H^{-1},\qquad
|a_i|< H^{r_i}\quad (1\le i\le n)
\end{equation}
is $a_0=\dots=a_n=0$.
\end{itemize}
\end{lemma}

The equivalence of (i) and (ii) is a straightforward consequence of the definition of $\Badr$. The equivalence of
(ii) and (iii) is relatively well known, see Appendix in \cite{Badziahin-Pollington-Velani-Schmidt} for a similar statement. Indeed, this equivalence is essentially a special case of Mahler's version of Khintchine's transference Principle appearing in \cite{Mahler-39:MR0001241}. To make this paper self-contained we provide further details in \ref{A}.

\subsection{Approximation by algebraic numbers of bounded degree}\label{2.3}

There are two classical interrelated settings in the theory of approximation by algebraic numbers of bounded degree. One of them boils down to investigating small values of integral polynomials $P$ with $\deg P\le n$ at a given number $\xi$. The other deals with the proximity of algebraic numbers $\alpha$ of degree $\le n$ to a given number $\xi$, see \cite{Bugeaud-04:MR2136100} for further background.
In particular, the long standing Wirsing--Schmidt conjecture \cite[p.258]{Schmidt-1980}, which was motivated by Wirsing's theorem \cite{Wirsing-60:MR0142510}, states that for any $n\in\N$ and any real transcendental number $\xi$ there is a constant $C=C(\xi,n)>0$ such that
$$
|\xi-\alpha|\le C(\xi,n) H(\alpha)^{-n-1}
$$
holds for infinitely many algebraic numbers $\alpha$ of degree $\le n$, where $H(\alpha)$ denotes the height of $\alpha$ (to be recalled a few lines below). The $n=1$ case of the conjecture is a trivial consequence of the theory of continued fractions. For $n=2$ it was proved by Davenport and Schmidt \cite{Davenport-Schmidt-67:MR0219476}. However, there are only partial results for $n>2$.
Note, however, that using Dirichlet's theorem it is easily shown that
for any $\xi\in\R$ there exists $c_0=c_0(\xi,n)>0$ such that $|P(\xi)|<c_0H(P)^{-n}$ for infinitely many $P\in\Z[x]$ with $\deg P\le n$.

In this section we will deal with real numbers badly approximable by algebraic numbers. Given a polynomial $P$ with integer coefficients, $H(P)$ will denote the height of $P$, which, by definition, is the maximum of the absolute values of the coefficients of $P$. Given an algebraic number $\alpha\in\C$, $H(\alpha)$ will denote the (naive) height of $\alpha$, which, by definition, is the height of the minimal defining polynomial $P$ of $\alpha$ over $\Z$. It is also convenient to introduce the following three sets:
\begin{align*}
\cB_n&=\left\{\xi\in\R:\begin{array}{l}
\exists\ c_1=c_1(\xi,n)>0\text{ such that }|P(\xi)|\ge c_1H(P)^{-n}\\
\text{for all non-zero }P\in\Z[x],\ \deg P\le n
               \end{array}
\right\},
\\[1ex]
\cW_n^*&=\left\{\xi\in\R:\begin{array}{l}
\exists\ c_2=c_2(\xi,n)>0\text{ such that }|\xi-\alpha|<c_2H(\alpha)^{-n-1}\\
\text{for infinitely many real algebraic }\alpha\text{ with }\deg\alpha\le n
               \end{array}
\right\},
\\[1ex]
\cB_n^*&=\left\{\xi\in\R:\begin{array}{l}
\exists\ c_3=c_3(\xi,n)>0\text{ such that }|\xi-\alpha|\ge c_3H(\alpha)^{-n-1}\\
\text{for all real algebraic }\alpha\text{ with } \deg\alpha\le n
               \end{array}
\right\}.
\end{align*}

The sets $\cB_n$ and $\cB_n^*$ are the natural generalisations of badly approximable numbers to the context of approximation by algebraic numbers. They are known to have Lebesgue measure zero, e.g., by a Khintchine type theorem proved in \cite{Beresnevich-99:MR1709049}.
Within this paper we will deal with the following two conjectures that Bugeaud formulated as Problems~24 and~25 in his Cambridge Tract \cite[\S10.2]{Bugeaud-04:MR2136100}:

\medskip

\noindent\textbf{Conjecture B1:} ~~\textit{$\cB_n$ contains a real transcendental number.}

\medskip

\noindent\textbf{Conjecture B2:} ~~\textit{$\cW_n^*\cap\cB_n^*$ contains a real transcendental number.}

\medskip

\noindent Note that Conjecture B1 is stronger than Conjecture B2 since we have that
\begin{equation}\label{incl}
\cB_n\subset\cW_n^*\cap\cB_n^*.
\end{equation}
The proof of \eqref{incl} is rather standard. Indeed, it rests on the Mean Value Theorem and Minkowski's theorem for convex bodies, see \ref{B} for details. Here we establish the following Hausdorff dimension result that easily settles the above conjectures.

\begin{theorem}\label{cor2}
For any natural number $n$ and any interval $I$ in $\R$
$$
\dim\bigcap_{k=1}^n\cB_k\cap I~=~\dim\bigcap_{k=1}^n(\cW_k^*\cap\cB_k^*\cap I)~=~1\,.
$$
\end{theorem}

\begin{proof}
Without loss of generality we will assume that $n\ge2$. Let
$$
\vv f:\R\to\R^n\qquad\text{such that}\qquad \vv f(x)=(x,x^2,\dots,x^n)\,,
$$
$1\le k\le n$ be an integer and $\rr_k=(\frac1k,\dots,\tfrac1k,0,\dots,0)\in\cR_n$, where the number of zeros is $n-k$. Let $\xi\in\R$ be such that $\vv f(\xi)\in\Bad(\rr_k)$. By Property~(iii) of Lemma~\ref{dual},
there exists $c(\xi,n,k)>0$ such that for any $H\ge1$ the only integer solution $(a_0,a_1,\dots,a_n)$ to the system
\begin{align*}
&|a_0+a_1x+\dots+a_nx^n|< c(\xi,n,k) H^{-1},\\
&|a_i|< H^{1/k}\quad (1\le i\le k),\\
&|a_i|< H^{0}\quad (k+1\le i\le n)
\end{align*}
is $a_0=\dots=a_n=0$. Hence, for any non-zero polynomial $P(x)=a_kx^k+\dots+a_0\in\Z[x]$ with $H(P)<H^{1/k}$ we must have that $|P(\xi)|\ge c(\xi,n,k) H^{-1}>c(\xi,n,k) H(P)^{-k}$. By definition, this means that $\xi\in\cB_k$. To sum up, we have just shown that $\vv f^{-1}(\Bad(\rr_k))\subset\cB_k$. Hence
$$
\bigcap_{k=1}^n\vv f^{-1}(\Bad(\rr_k))~~\subset~~\bigcap_{k=1}^n\cB_k\stackrel{\eqref{incl}}{~~\subset~~}
\bigcap_{k=1}^n\cW_k^*\cap\cB_k^*\,.
$$
By Theorem~\ref{t2}, for any interval $I\subset\R$ we have that $\dim \bigcap_{k=1}^n\vv f^{-1}(\Bad(\rr_k))\cap I=1$. In view of the above inclusions the statement of Theorem~\ref{cor2} now readily follows.
\end{proof}

\bigskip

\noindent\textit{Remark.} An interesting problem is to show that Theorem~\ref{cor2} holds when $n=\infty$.

\section{Lattice points counting}\label{counting}

The rest of the paper will be concerned with the proof of Theorem~\ref{t2}, which
will rely heavily on efficient counting of lattice points in convex bodies. The lattices will arise upon reformulating $\Badr$ in the spirit of Dani \cite{Dani-85:MR794799} and Kleinbock \cite{Kleinbock-98:MR1646538}.
This will require the following notation.
Given a subset $\Lambda$ of $\R^{n+1}$, let
\begin{equation}\label{delta}
\delta(\Lambda)=\inf_{\vv a\in\Lambda\setminus\{\vv0\}}\|\vv a\|_\infty\,,
\end{equation}
where $\|\vv a\|_\infty=\max\{|a_0|,\dots,|a_n|\}$ for $\vv a=(a_0,\dots,a_n)$.  Given $0<\kappa<1$, let
\begin{equation}\label{eq1}
\Gky=\left(\begin{array}{cc}
             \kappa^{-1} & \kappa^{-1}\vv y \\[1ex]
             0 & I_n
           \end{array}
\right),
\end{equation}
where $\vv y\in\R^n$ is regarded as a row and $I_n$ is the $n\times n$ identity matrix. Finally, given
$\rr\in\cR_n$, $b>1$ and $t\in\R$, define the $(n+1)\times(n+1)$ unimodular diagonal matrix
\begin{equation}\label{gtr}
\gtr={\rm diag}\{\bb^{t},\bb^{-r_1t},\dots,\bb^{-r_nt}\}.
\end{equation}

\begin{lemma}\label{lemma08}
Let $\vv y\in\R^n$, $\rr\in\cR_n$. Then
$\vv y\in\Badr$ if and only if there exists $\kappa\in(0,1)$ and $b>1$ such that for all\/ $t\in\N$
\begin{equation}\label{dani}
\delta(\gtr \Gky\Z^{n+1})\ge1.
\end{equation}
\end{lemma}

\begin{proof}
The necessity is straightforward as all one has to do is to take $H=\bb^t$ and divide each inequality in (\ref{eee10}) by its right hand side. Then, assuming that $\vv y\in\Badr$, the non-existence of integer solutions to (\ref{eee10}) would imply (\ref{dani}) with $\kappa=c$.
The sufficiency is only slightly harder. Assume that for some $\kappa$ and $b$ inequality (\ref{dani}) holds for all\/ $t\in\N$, while $\vv y\not\in\Badr$. Take $c=\kappa/b$. By definition, there is an $H>1$ such that (\ref{eee10}) has a non-zero integer solution $(a_0,\dots,a_n)$. Take $t=[\log H/\log \bb]+1$, where $[\,\cdot\,]$ denotes the integer part. Note that $H\bb^{-t}<1$ and $H^{-1}b^t\le \bb$. Then (\ref{eee10}) implies that $\delta(\gtr \Gky\Z^{n+1})<1$, contrary to (\ref{dani}). The proof is thus complete.
\end{proof}

\bigskip

\noindent\textit{Remark.}
Lemma~\ref{lemma08} can be regarded as a variation of the Dani-Kleinbock correspondence between badly approximable points in $\R^n$ and bounded orbits of certain lattices under the actions by the diagonal semigroup $\{g^t_{\rr,b}:t>0\}$, where $b>1$. It is easily seen that this semigroup is independent of the choice of $b>1$, which is usually taken to be $e = \exp(1)$. The correspondence was first established by Dani \cite{Dani-85:MR794799} in the case $\rr=(\tfrac1n,\dots,\tfrac1n)$ and then extended by Kleinbock \cite{Kleinbock-98:MR1646538} to the case of arbitrary positive weights and can be stated as follows. The point $\vv y\in\R^n$ is $\rr$-badly approximable if and only if the orbit of the lattice $G(1;\vv y)\Z^{n+1}$ under the action by $\{g^t_{\rr,e}:t>0\}$ is bounded.

\medskip

We proceed by recalling two classical results from the geometry of numbers. In what follows, $\vol_\ell(X)$ denotes the $\ell$-dimensional volume of $X\subset\R^\ell$ and $\#X$ denotes the cardinality of $X$. Also $\det\Lambda$ will denote the {\em determinant \em or \em covolume} of a lattice $\Lambda$.

\begin{Minkowski}
Let $K\subset\R^\ell$ be a convex body symmetric about the origin and let $\Lambda$ be a lattice in\/ $\R^\ell$. Suppose that $\vol_\ell(K)>2^\ell\det\Lambda$. Then $K$ contains a non-zero point of $\Lambda$.
\end{Minkowski}

\begin{Blichfeldt}
Let $K\subset\R^\ell$ be a convex bounded body containing $\vv0$ and let $\Lambda$ be a lattice in\/ $\R^\ell$ such that $\rank(K\cap\Lambda)=\ell$. Then
$$
\#(K\cap\Lambda)\,\le \,\ell!\,\frac{\vol_\ell(K)}{\det\Lambda}+\ell.
$$
\end{Blichfeldt}

\noindent The following lemma is a straightforward consequence of Blichfeldt's theorem.

\begin{lemma}[cf. Lemma~4 in \cite{Kristensen-Thorn-Velani-06:MR2231044}]\label{lemma09}
Let $K$ be a convex bounded body in $\R^\ell$ with $\vv0\in K$ and $\vol_\ell(K)<1/\ell!$. Then $\rank(K\cap\Z^\ell)\le \ell-1$.
\end{lemma}

\begin{proof}
Assume the contrary, that is assume that $\rank(K\cap\Z^\ell)=\ell$ (note that the rank cannot be bigger than $\ell$). It means that $K$ contains at least $\ell$ non-zero integer points. Since $\vv0\in K$, we then have that $\#(K\cap\Z^\ell)\ge\ell+1$. However, since $\det\Z^\ell=1$ and $\vol_\ell(K)<1/\ell!$, by Blichfeldt's theorem, we conclude that
$$
\#(K\cap\Z^\ell)\,\le \,\ell!\,\frac{\vol_\ell(K)}{\det\Lambda}+\ell< \ell!\,\frac{1/\ell!}{1}+\ell<1+\ell,
$$
contrary to the above lower bound.
\end{proof}

\bigskip

\noindent The bodies $K$ of interest will arise as the intersection of parallelepipeds
\begin{equation}\label{eq91}
\Pi_{\bm\theta}=\big\{\vv x=(x_0,\dots,x_n)\in\R^{n+1}:|x_i|< \theta_i, ~~i=0,\dots,n\big\}
\end{equation}
with $\ell$-dimensional subspaces of $\R^{n+1}$, where $\bm\theta=(\theta_0,\dots,\theta_n)$ is an $(n+1)$-tuple of positive numbers. In view of this, we now obtain an estimate for the volume of the bodies that arise this way (Lemma~\ref{lemma10} below) and then verify what Blichfeldt's theorem means for such bodies (Lemma~\ref{lemma11} below).

\begin{lemma}\label{lemma10}
Let $\ell\in\N$, $\ell\le n+1$, $\bm\theta=(\theta_0,\dots,\theta_n)$ with $\theta_0,\dots,\theta_n>0$.
Then for any linear subspace $V$ of\/ $\R^{n+1}$ of dimension $\ell$ we have that
$$
\vol_\ell(\Pi_{\bm\theta}\cap V)\le 2^\ell(n+1)^{\ell/2}\Theta_\ell,
\qquad\text{where}\qquad
\Theta_\ell=\max_{\substack{I\subset\{0,\dots,n\}\\ \#I=\ell}}\prod_{i\in I}\theta_i.
$$
\end{lemma}

\begin{proof}
Since $V$ is a linear subspace of $\R^{n+1}$ of dimension $\ell$, it is given by $n+1-\ell$ linear equations. Using Gaussian elimination, we can rewrite these equations to parametrise $V$ with a linear map $\vv f:\R^\ell\to\R^{n+1}$ of $x_{i_1},\dots,x_{i_\ell}$ such that
$$
\vv f(x_{i_1},\dots,x_{i_\ell})=(x_{i_1},\dots,x_{i_\ell})M,
$$
where $M=(m_{i,j})$ is an $\ell\times (n+1)$ matrix with $|m_{i,j}|\le1$ for all $i$ and $j$. Then note that $\vol_\ell(\Pi_{\bm\theta}\cap V)$ is bounded by the area of the intersection of $V$ with the cylinder $|x_{i_j}|\le\theta_{i_j}$ for $j=1,\dots, \ell$. This area is equal to
\begin{equation}\label{eee11}
\int_{-\theta_{i_1}}^{\theta_{i_1}}\cdots \int_{-\theta_{i_\ell}}^{\theta_{i_\ell}} \left\|\frac{\partial \vv f}{\partial x_{i_1}}\wedge\ldots\wedge \frac{\partial \vv f}{\partial x_{i_\ell}}\right\|_e dx_{i_1}\ldots dx_{i_\ell}\,,
\end{equation}
where $\|\cdot\|_e$ is the Euclidean norm on $\bigwedge^\ell(\R^{n+1})$.
Since $|m_{i,j}|\le 1$, every coordinate of every partial derivative of $\vv f$ is bounded by $1$ in absolute value. Hence $\|\partial\vv f/\partial x_{i_j}\|_e\le\sqrt{n+1}$ and the integrand in (\ref{eee11}) is bounded above by $(\sqrt{n+1})^{\ell}$. This readily implies that the area given by (\ref{eee11}) is bounded above by $2^\ell(n+1)^{\ell/2}\theta_{i_1}\cdots\theta_{i_\ell}\le2^\ell(n+1)^{\ell/2}\Theta_\ell$, whence the result follows.
\end{proof}

\begin{lemma}\label{lemma11}
Let $c(n)=4^{n+1}(n+1)^{(n+1)/2}(n+1)!$ and let $\bm\theta$ and $\Theta_\ell$ be as in Lemma~\ref{lemma10}. Then for any discrete subgroup $\Gamma$ of\/ $\R^{n+1}$ with $\ell=\rank\big(\Gamma\cap\Pi_{\bm\theta}\big)>0$ we have that
\begin{equation}\label{eee12}
\#\big(\Gamma\cap\Pi_{\bm\theta}\big)\ \le \ c(n)\,\frac{\Theta_{\ell}}{\delta(\Gamma)^\ell}+n+1.
\end{equation}
\end{lemma}

\begin{proof}
Let $V=\spn(\Gamma\cap\Pi_{\bm\theta})$ and $\Lambda=V\cap\Gamma$. Clearly, $\rank(\Lambda)=\ell$ and furthermore $\Lambda$ is a lattice in $V$. Also note that $\Gamma\cap\Pi_{\bm\theta}=\Lambda\cap\Pi_{\bm\theta}$. Since $\Lambda\subseteq\Gamma$, we have that $\delta(\Gamma)\le\delta(\Lambda)$. Let $B(r)$ denote the open ball in $V$ of radius $r$ centred at the origin. Note that the length of any non-zero point in $\Lambda$ is bigger than or equal to $\delta(\Lambda)\ge \delta(\Gamma)$. Hence, by Minkowski's convex bodies theorem, we must have that $\vol_\ell\big(B(\delta(\Gamma)\big)\le 2^\ell\det\Lambda$, whence we obtain $\det\Lambda\ge \vol_\ell(B(\delta(\Lambda)))2^{-\ell}\ge (\delta(\Lambda)/2)^{\ell}$. Now using this inequality, Blichfeldt's theorem, Lemma~\ref{lemma10} and the fact that $\ell\le n+1$ readily gives (\ref{eee12}).
\end{proof}

\bigskip

We are now approaching the key counting result of this section. Let
\begin{equation}\label{eq92}
\Pi(b,u)\stackrel{\rm def}{=} \Pi_{\bm\theta}\qquad\text{with ~~$\bm\theta=(\bb^{u},1,\dots,1)$},
\end{equation}
where $u>0$, $b>1$ and $\Pi_{\bm\theta}$ is given by (\ref{eq91}). Given $\rr\in\cR_n$, let
\begin{equation}\label{eq12}
z(\rr) \stackrel{\rm def}{=}\#\{\,i\,:\,r_i=0\,\}\qqand
\lambda(\rr) \stackrel{\rm def}{=} \big(1+\tau(\rr)\big)^{-1}.
\end{equation}
Recall that $\tau(\rr)$, $\delta(\cdot)$, $\gtr$ and $\Pi(\bb,u)$ are given by (\ref{eee03}), \eqref{delta}, \eqref{gtr} and \eqref{eq92} respectively, and $[x]$ denotes the integer part of $x$.

\begin{lemma}\label{lemma12}
Let\/ $b>1$, $\rr\in\cR_n$, $\lambda=\lambda(\rr)$, $z=z(\rr)$, $t\in\N$, $u\in\R$, $1\le \lambda u\le t$ and
$c(n)$ be as in Lemma~\ref{lemma11}. Let $g^t=\gtr$. Let $\Lambda$ be a discrete subgroup of\/ $\R^{n+1}$ such that $\rank\Lambda\le n-z$ and
\begin{equation}\label{eee13}
\delta\big(g^{t-[\lambda u]}\Lambda\big)\ge1.
\end{equation}
Then
\begin{equation}\label{eee13b}
\#(g^t\Lambda)\cap\Pi(b,u)\le 2c(n) \bb^{\tau}\bb^{\lambda u}.
\end{equation}
\end{lemma}

\begin{proof}
Let $\vv x=(x_0,\dots,x_n)\in \Lambda$ be such that $\gt\vv x\in\Pi(b,u)$. By the definitions of $\gt=\gtr$ and $\Pi(b,u)$, we have that $\bb^{t}|x_0|<\bb^{u}$ and $\bb^{-r_it}|x_i|<1$ for $i=1,\dots,n$. Equivalently, for $s\in\Z$, $1\le s\le u-1$, we have that
$$
\bb^{t-s}|x_0|<\bb^{u-s}\qqand \bb^{-r_i(t-s)}|x_i|<\bb^{r_is}\quad (1\le i\le n).
$$
This can be written as
$g^{t-s}\vv x\in\Pi_{\bm\theta}$, where $\bm\theta=(\bb^{u-s},\bb^{r_1s},\dots,\bb^{r_ns})$.
Therefore,
\begin{equation}\label{eee14}
(\gt\Lambda)\cap\Pi(b,u)\ =\ \Gamma\cap\Pi_{\bm\theta},
\end{equation}
where $\Gamma=g^{t-s}\Lambda$. Now take $s=[\lambda u]$. Recall that $\lambda<1$ and that, by the conditions of Lemma~\ref{lemma12}, $[\lambda u]\le t$. Then, by the left hand side of (\ref{eee13}), we have that $\delta(\Gamma)\ge1$. Hence, by Lemma~\ref{lemma11} and (\ref{eee14}), we get
\begin{equation}\label{eee15}
\#(\gt\Lambda)\cap\Pi(b,u)\ = \ \#\big(\Gamma\cap\Pi_{\bm\theta}\big)\ \le\ c(n)\Theta_\ell+n+1,
\end{equation}
where $\ell=\rank\Gamma=\rank\Lambda\le n-z$. Note that all the components of $\bm\theta$ are $\ge1$ and exactly $z$ of them equal $1$. Then, since $\ell\le n-z$ and $s=[\lambda u]$, we get that
$$
\Theta_\ell \ \le\ \frac{\theta_0\dots\theta_n}{\min\{\theta_i:\theta_i>1\}}= \frac{\bb^u}{\min\{\bb^{u-s},\bb^{\tau s}\}}\le \max\{\bb^{\lambda u},\bb^{u-\tau(\lambda u-1)}\}=\bb^{\tau}\bb^{\lambda u}.
$$
Combining this estimate with (\ref{eee15}) and the obvious fact that $n+1<c(n)\bb^{\tau}\bb^{\lambda u}$ gives (\ref{eee13b}).
\end{proof}

\section{`Dangerous' intervals}\label{dan}

In view of Lemma~\ref{lemma08}, when proving Theorem~\ref{t2} we will aim to avoid the solutions of the inequalities $\delta(\gtr \Gx\Z^{n+1})<1$, where $\Gx=\Gky$ with $\vv y=\vv f(x)$ and $\kappa$ is a sufficiently small constant. For fixed $\rr,\bb,t,\vv f$ and $\kappa$ the above inequality is equivalent to the existence of $(a_0,\vv a)\in\Z^{n+1}$ with $\vv a\neq\vv0$ satisfying
\begin{equation}\label{eq+08}
\left\{    \begin{array}{rcl}
                             |a_0+\vv a.\vv f(x)| &<& \kappa \bb^{-t},  \\[0.4ex]
                              |a_i|&<& \bb^{r_it} \ \ (1\le i\le n).
                           \end{array}
\right.
\end{equation}
Here the dot means the usual inner product. That is $\vv a.\vv b=a_1b_1+\dots+a_nb_n$ for any given $\vv a=(a_1,\dots,a_n)$ and $\vv b=(b_1,\dots,b_n)$. In this section we study intervals arising from (\ref{eq+08}) that, for obvious reasons, are referred to as \emph{dangerous} (see \cite{Schmidt-1980} for similar terminology). We will consider several cases that are tied up with the magnitude of $\vv a.\vv f'(x)$; \emph{i.e.}, the derivative of $a_0+\vv a.\vv f(x)$, -- see Propositions~\ref{prop0} and \ref{prop1} below.

Throughout $\cF_n(I)$ and $x_0$ are as in Theorem~\ref{t2}. First we discuss some conditions that arise from the nondegeneracy assumption on maps in $\cF_n(I)$. Let $\vv f=(f_1,\dots,f_n)\in\cF_n(I)$. Since $\vv f$ is nondegenerate at $x_0\in I$, there is a sufficiently small neighborhood $I_{\vv f}$ of $x_0$ such that the Wronskian of $f'_1,\dots,f'_n$, which, by definition, is the determinant $\det\big(f^{(i)}_{j}\big)_{1\le i,j\le n}$,  is non zero everywhere in $I_{\vv f}$. Then every coordinate function $f_{j}$ is non-vanishing at all but countably many points of $I_{\vv f}\subset I$ -- see, e.g., \cite[Lemma~3]{Beresnevich-Bernik-96:MR1387861}. Since $\vv f\in C^n$ and $\cF_n(I)$ is finite, we can choose a compact interval $I_0\subset \bigcap_{\vv f\in\cF_n(I)}I_{\vv f}\subset I$ satisfying

\medskip

\noindent\textbf{Property~F:} There are constants $0<c_0<1<c_1$ such that for every map $\vv f=(f_1,\dots,f_n)\in\cF_n(I)$, for all $x\in I_0$, $1\le i\le n$ and $0\le j\le n$ one has that
\begin{equation}\label{eee05}
\left|\det\big(f^{(i)}_{j}(x)\big)_{1\le i,j\le n}\right|>c_0,
\qquad |f'_{j}(x)|>c_0\qqand |f^{(i)}_{j}(x)|< c_1.
\end{equation}

\medskip

\noindent Next, we prove two auxiliary lemmas that are well known in a related context.

\begin{lemma}[cf. Lemma~5 in \cite{Beresnevich-Bernik-96:MR1387861}]\label{lemma13}
Let $I_0\subset I$ be a compact interval satisfying Property~F. Let $2c_2=c_0c_1^{-n+1}n!^{-1}$, where $c_0$ and $c_1$ arise from \eqref{eee05}. Then for any $\vv f\in\cF_n(I)$, any $\vv a=(a_1,\dots,a_n)\in\Z^n\setminus\{0\}$ and any $x\in I_0$
there exists $i\in\{1,\dots,n\}$ such that
$
|\vv a.\vv f^{(i)}(x)| \ge 2c_2\max_{1\le j\le n}|a_j|.
$
\end{lemma}

\begin{proof}
Solving the system $a_1f^{(i)}_{1}(x)+\dots+a_nf^{(i)}_{n}(x)=\vv a.\vv f^{(i)}(x)$, where $1\le i\le n$, by Cramer's rule with respect to $a_i$ and using (\ref{eee05}) to estimate the determinants involved in the rule we obtain
$$
|a_j|\ \le\ c_1^{n-1}\cdot n!c_0^{-1}\max_{1\le i\le n}|\vv a.\vv f^{(i)}(x)|
$$
for each $j=1,\dots,n$, whence the statement of lemma readily follows.
\end{proof}

\begin{lemma}[cf. Lemma~6 in \cite{Beresnevich-Bernik-96:MR1387861}]\label{lemma14}
Let $I_0\subset I$ and $c_2$ be as in Lemma~\ref{lemma13}. Then
there is $\delta_0>0$ such that for any interval $J\subset I_0$ of length $|J|\le\delta_0$, any $\vv f\in\cF_n(I)$ and $\vv a=(a_1,\dots,a_n)\in\Z^n\setminus\{0\}$, there is an $i\in\{1,\dots,n\}$ satisfying
\begin{equation}\label{eee16}
\inf_{x\in J}|\vv a.\vv f^{(i)}(x)|\ge c_2\max_{1\le j\le n}|a_j|.
\end{equation}
\end{lemma}

\begin{proof}
Since $I_0$ is compact, for each $\vv f\in\cF_n(I)$ and $1\le i\le n$, the map $\vv f^{(i)}$ is uniformly continuous on $I_0$. Hence, there is a $\delta_{i,\vv f}>0$ such that for any $x,y\in I_0$ with $|x-y|\le\delta_{i,\vv f}$ we have $|\vv f^{(i)}(x)-\vv f^{(i)}(y)|<c_2/n$. Let $J\subset I_0$ be an interval of length $|J|\le \delta_{i,\vv f}$ and $x,y\in J$. By Lemma~\ref{lemma13}, there is $i\in\{1,\dots,n\}$ such that
$|\vv a.\vv f^{(i)}(x)|\ \ge\ 2c_2h$, where $h=\max_{1\le j\le n}|a_j|$.
Then
\begin{equation}\label{eee17}
|\vv a.\vv f^{(i)}(y)|\ge |\vv a.\vv f^{(i)}(x)|-|\vv a.\vv f^{(i)}(y)-\vv a.\vv f^{(i)}(x)| \ge 2c_2h-nh c_2/n=c_2h.
\end{equation}
Since $\cF_n(I)$ is finite, $\delta_0=\inf_{i,\vv f}\delta_{i,\vv f}>0$. Hence (\ref{eee17}) implies (\ref{eee16}) provided that $|J|\le \delta_0$.
\end{proof}

\medskip

\begin{proposition}\label{prop0}
Let $I_0\subset I$ be a compact interval satisfying Property~F and $\vv f\in\cF_n(I)$. Further, let  $\delta_0$ be as in Lemma~\ref{lemma14}, $\rr\in\cR_n$ and\\[-1.5ex]
\begin{equation}\label{eee18}
    \gamma=\gamma(\vv r)\, \stackrel{\rm def}{=} \, \max\{r_1,\dots,r_n\}.
\end{equation}
Finally, let $t\in\N$, $\ell\in\Z_{\ge0}$, $\bb>1$, $\vv a\in\Z^n\setminus\{\vv0\}$, $a_0\in\Z$, $0<\kappa<1$ and
$$
D^1_{t,\ell,\rr,\bb,\kappa,\vv f}(a_0,\vv a)=\left\{x\in I_0~:~\begin{array}{rcl}
                             |a_0+\vv a.\vv f(x)| &<& \kappa \bb^{-t}  \\[0.4ex]
                              \bb^{\gamma t-(1+\gamma)\ell}\le~ |\vv a.\vv f'(x)| &<& \bb^{\gamma t-(1+\gamma)(\ell-1)} \\[0.4ex]
                              |a_i|&<& \bb^{r_it}\\
                           \end{array}
\right\}.
$$
Then, there is a constant $c_3>0$ depending on $n$, $|I_0|$, $c_1$, $c_2$ and $\delta_0$ only such that the set $D^1_{t,\ell,\rr,\bb,\kappa,\vv f}(a_0,\vv a)$ can be covered by a collection $\cD^1_{t,\ell,\rr,\bb,\kappa,\vv f}(a_0,\vv a)$ of at most $c_3$ intervals $\Delta$ of length
$
|\Delta|\le \kappa \bb^{-(1+\gamma)(t-\ell)}.
$
\end{proposition}

\begin{proof}
We will abbreviate $D^1_{t,\ell,\rr,\bb,\kappa,\vv f}(a_0,\vv a)$ as $D^1_{\pp}$ and naturally assume that $D^1_{\pp}\neq\emptyset$ as otherwise there is nothing to prove.
Since $I_0$ can be covered by at most $[\delta_0^{-1}|I_0|]+1$ intervals $J$ of length $|J|\le\delta_0$, it suffices to prove the proposition under the assumption that $|I_0|\le\delta_0$. Let $f(x)=a_0+\vv a.\vv f(x)$. Then, by Lemma~\ref{lemma14}, we have that  $|f^{(i)}(x)|>0$ for a fixed $i\in\{1,\dots,n\}$ and all $x\in I_0$. First consider the case $i>1$. Then, using Rolle's theorem, one finds that the function $f^{(j)}(x)$ vanishes on $I_0$ at $\le i-j$ points $(0\le j\le i-1)$. Assuming that $I_0=[a,b]$, let
$x_0=a<x_1<\dots<x_{s-1}<x_s=b$\/ be the collection consisting of the points
$a$ and $b$ and all the zeros of $\prod_{j=0}^{i-1}f^{(j)}(x)$. Then, as we have just seen $s\le
1+\sum_{j=0}^{i-1}(i-j)= i(i+1)/2+1$. By the choice of the points $x_i$, we have that for
$1\le q\le s$ and $0\le j\le i-1$ the function $f^{(j)}(x)$
is monotonic and does not change sign on the interval $[x_{q-1},x_q]$. Therefore, in view of the definition of $D^1_{\pp}$ we must have that $\Delta_q=D^1_{\pp}\cap[x_{q-1},x_q]$ is an interval. Hence,
$D^1_{\pp}=\bigcup_{q=1}^s\Delta_q$, a union of at most
$(i+1)i/2+1\le (n+1)n/2+1$ intervals.

\smallskip

It remains to estimate the length of each $\Delta_q$. To this end, take any $x_1,x_2\in \Delta_q$. By the construction of $\Delta_q$, the numbers $f(x_1)$ and $f(x_2)$ have the same sign and satisfy the inequality $|f(x_i)|<\kappa \bb^{-t}$. Hence, $|f(x_1)-f(x_2)|<\kappa \bb^{-t}$. By the Mean Value Theorem, $|f(x_1)-f(x_2)|=|f'(\theta)(x_1-x_2)|$. Hence $|x_1-x_2|\le \kappa \bb^{-t}/|f'(\theta)|$. Since $\Delta_q\subset D^1_{\pp}$ is an interval, $\theta\in D^1_{\pp}$. Hence, $|f'(\theta)| \ge \bb^{\gamma t-(1+\gamma)\ell}$ and we obtain that $|x_1-x_2|\le \kappa \bb^{-t} \bb^{-\gamma t+(1+\gamma)\ell}=\kappa \bb^{-(1+\gamma)(t-\ell)}$. This estimate together with the obvious equality $|\Delta_q|=\sup_{x_1,x_2\in\Delta_q}|x_1-x_2|$ implies that $|\Delta_q|\le \kappa \bb^{-(1+\gamma)(t-\ell)}$. Thus, if $i > 1$, the set $D^1$ can be covered by at most $n(n+1)/2+1$ intervals of length $\kappa \bb^{-(1+\gamma)(t-\ell)}$.

\smallskip

Now consider the case $i=1$. Recall that $f(x)=a_0+\vv a.\vv f(x)$. Then, by the definition of $D^1_{\pp}$ and (\ref{eee05}), for $x\in D^1$ we get
\begin{equation}\label{d+}
\bb^{\gamma t-(1+\gamma)\ell}\le |f'(x)|=|\vv a.\vv f'(x)|\le c_1n\max_{1\le j\le n}|a_j|.
\end{equation}
Further, (\ref{eee16})${}_{i=1}$ implies that $\inf_{x\in I_0}|f'(x)|\ge c_2\max_{1\le j\le n}|a_j|$. Therefore, $f$ is monotonic on $I_0$ and $D^1_{\pp}$ is covered by a single interval $\Delta$ defined by the inequality $|f(x)|<\kappa \bb^{-t}$. Arguing as above and using (\ref{d+}) we get
$$
\begin{array}{rcl}
|\Delta| &\le& \displaystyle\frac{2\kappa \bb^{-t}}{\inf_{x\in I_0}|f'(x)|}\le \frac{2\kappa \bb^{-t}}{c_2\max_{1\le j\le n}|a_j|}\\[4ex]
& \le & \displaystyle\frac{2c_1n\kappa \bb^{-t}}{c_2\bb^{\gamma t-(1+\gamma)\ell}}=\frac{2c_1n}{c_2}\times \kappa \bb^{-(1+\gamma)(t-\ell)}.\\[1ex]
\end{array}
$$
Thus, by splitting $\Delta$ into smaller intervals if necessary, $D^1_{\pp}$ can be covered by at most $[\frac{2c_1n}{c_2}]+1$ intervals of length $\kappa \bb^{-(1+\gamma)(t-\ell)}$.
\end{proof}

\begin{proposition}\label{prop1}
Let $I_0\subset I$ be a compact interval satisfying Property~F and $\gamma=\gamma(\rr)$ be given by \eqref{eee18}. Then there are constants $K_0>0$ and $0<\kappa_0<1$ such that for any $\vv f\in\cF_n(I)$, any $\rr\in\cR_n$, $t\in\N$, $0\le \ve<\gamma$, $\bb>1$ and $0<\kappa<\kappa_0$ the set
$$
D^2_{t,\ve,\rr,b,\kappa,\vv f}=\left\{x\in I_0:\exists\, \vv a\in\Z^n\setminus\{0\} \text{ and }a_0\in\Z\ \text{such that}\begin{array}{l}
                             |a_0+\vv a.\vv f(x)| < \kappa\, \bb^{-t}  \\[0.3ex]
                              |\vv a.\vv f'(x)| < nc_1\bb^{(\gamma-\ve) t} \\[0.3ex]
                              |a_i|< \bb^{r_it}
                           \end{array}
\right\}
$$
can be covered by a collection $\cD^2_{t,\ve,\rr,b,\kappa,\vv f}$ of intervals such that
\begin{equation}\label{eq5}
|\Delta|\le \delta_t\qquad\text{for all }~\Delta\in \cD^2_{t,\ve,\rr,b,\kappa,\vv f}
\end{equation}
and\\[-3ex]
\begin{equation}\label{eq6}
\#\cD^2_{t,\ve,\rr,b,\kappa,\vv f}\le \frac{K_0\,(\kappa\,\bb^{-\ve t})^{\alpha}}{\delta_t}\,,
\end{equation}
where
$\delta_t=\kappa\,\bb^{-t(1+\gamma-\ve)}$ ~and~ $\alpha=\tfrac1{(n+1)(2n-1)}$.
\end{proposition}

Proposition~\ref{prop1} will be derived from a theorem due to Bernik, Kleinbock and Margulis using the ideas of \cite{Beresnevich-Bernik-Dodson-02:MR2069553}. In what follows $|X|$ denotes the Lebesgue measure of a set $X\subset \R$.
The following is a simplified version of Theorem~1.4 from \cite{Bernik-Kleinbock-Margulis-01:MR1829381} that refines the results of \cite{Kleinbock-Margulis-98:MR1652916}.

\begin{theorem}[Theorem~1.4 in \cite{Bernik-Kleinbock-Margulis-01:MR1829381}]\label{thmBKM}
Let $I\subset\R$ be an open interval, $x_0\in I$ and $\vv f\,:\,I\to \R^n$ be nondegenerate at~$x_0$. Then there is  an open interval $J\subset I$ centred at $x_0$ and $E_J>0$ such that for any real $\omega,K,T_1,\dots,T_n$ satisfying
$$
0<\omega\le 1,\quad T_1,\ldots,T_n\ge 1,\quad
K>0\quad\mbox{ and }\quad\omega KT_1\cdots T_n\le\max_i T_i
$$
the set\\[-3ex]
$$
S(\omega,K,T_1,\dots,T_n)\stackrel{\rm def}{=} \left\{x\in I\,:\, \exists\ \vv a\in \Z^n \backslash \{0\}~~\left. \begin{array}{l} \,\|\vv a.\vv f(x)\|<\omega\\[0.3ex]
~|\vv a.\vv f'(x)|<K\\[0.3ex]
\,|a_i|<T_i \quad (1\le i\le n)
\end{array}\right.     \right\}
$$
satisfies\\[-2ex]
\begin{equation}\label{vb023}
|S(\omega,K,T_1,\dots,T_n)\cap J|\ \le \ E_J\cdot\max\left(\omega, \left(\frac{\omega KT_1\cdots T_n}{\max_i
T_i}\right)^{\frac{1}{n+1}}\right)^{\frac1{2n-1}}.
\end{equation}
\end{theorem}

We will also use the following elementary consequence of Taylor's formula.

\begin{lemma}\label{lemma15}
Let $f:J\to\R$ be a $C^2$ function on an interval $J$. Let $\omega,K>0$ and $y\in J$ be such that $|f''(x)|<K^2/\omega$ for all $x\in J$ and
\begin{equation}\label{1}
    |f(y)|<\omega/2\qqand |f'(y)|<K/2\,.
\end{equation}
Then ~$|f(x)|<\omega$~ and ~$|f'(x)|<K$~ for all $x\in J$ with $|x-y|<\omega/2K$.
\end{lemma}

%
%

\noindent\textit{Proof of Proposition~\ref{prop1}.}\/
Fix any $\vv f\in\cF_n(I)$. We will abbreviate $D^2_{t,\ve,\rr,\bb,\kappa,\vv f}$ as $D^2_{\ppp}$ and naturally assume that it is non-empty as otherwise there is nothing to prove. By (\ref{eee05}), $\vv f$ is nondegenerate at any $x\in I_0$ and therefore Theorem~\ref{thmBKM} is applicable. Let $J=J(x)$ be the interval centred at $x$ that arises from Theorem~\ref{thmBKM}. Since $I_0$ is compact there is a finite cover of $I_0$ by intervals $J(x_1),\dots, J(x_s)$, where $s=s_{\vv f}$ depends on $\vv f$.
Let $0<\kappa_0<1$ and $\kappa_0\le \min_{1\le i\le s}|I_0\cap J(x_i)|$.
The existence of $\kappa_0$ is obvious because $|I_0\cap J(x)|>0$ for each $x\in I_0$.

Let $0<\kappa<\kappa_0$, $\rr\in\cR_n$, $t\in\N$, $0\le \ve<\gamma$, $\bb>1$ and let
\begin{equation}\label{eq32}
\omega=2\kappa\,\bb^{-t},\qquad K=2nc_1\bb^{(\gamma-\ve)t}\qqand T_i=\bb^{r_i t} ~~~(1\le i\le n).
\end{equation}
Note that since $\ve<\gamma$ and $c_1>1$ we have that $K>2$. Also note that $\omega<2\kappa$. For each $i\in\{1,\dots,s\}$ define the interval $J_i=(a_i+\omega/2K,b_i-\omega/2K)$, where $[a_i,b_i]$ is the intersection of $I_0$ and the closure of $J(x_i)$. Since $\kappa<\kappa_0\le |I_0\cap J(x_i)|$, $\omega<2\kappa$ and $K>2$, we have that $J_i\neq\emptyset$ for each $i$. Let
\begin{equation}\label{eq2}
\textstyle\tilde D^2_{\ppp}=\bigcup\limits_{1\le i\le s}\ \bigcup\limits_{y\in D^2_{\ppp}\cap J_i}\big(y-\omega/2K,y+\omega/2K\big).
\end{equation}
Our goal now is to use Lemma~\ref{lemma15} with $f(x)=a_0+\vv a.\vv f(x)$ in order to show that
\begin{equation}\label{eq31}
\textstyle\tilde D^2_{\ppp} \subset \bigcup\limits_{1\le i\le s} S(\omega,K,T_1,\dots,T_n)\cap J(x_i).
\end{equation}
In view of the definitions of $D^2_{\ppp}$ and $S(\omega,K,T_1,\dots,T_n)$ and the choice of parameters (\ref{eq32}), inequalities (\ref{1}) hold for every $y\in D^2$. Further, by (\ref{eee05}), the inequalities $|a_i|< \bb^{r_it}$ and the fact that $r_i\le\gamma$ for all $i$ implied by (\ref{eee18}), we get that
\begin{equation}\label{eq+14}
|f''(x)|\le nc_1\max_{1\le j\le n}|a_j|\le nc_1\max_{1\le j\le n}\bb^{r_it} \le nc_1\bb^{\gamma t}.
\end{equation}
Next, $K^2/\omega=\frac12n^2c_1^2\kappa^{-1}\bb^{2(\gamma-\ve)t}\bb^t>nc_1\bb^{\gamma t}$ because $\ve\le \gamma\le 1$, $c_1>1$ and $\kappa<1$. Therefore, by (\ref{eq+14}), we have that $|f''(x)|\le K^2/\omega$ for all $x\in I_0$. Thus, Lemma~\ref{lemma15} is applicable and for $1\le i\le s$ we have that $\{x:|x-y|<\omega/2K\}\subset S(\omega,K,T_1,\dots,T_n)\cap J(x_i)$ each $y\in D^2_{\ppp}\cap J_i$. This proves (\ref{eq31}).

Next, by Theorem~\ref{thmBKM}, condition $r_1+\dots+r_n=1$ and (\ref{eq31}) we conclude that
\begin{equation}\label{eq3}
|\tilde D^2_{\ppp}|\le E_{\vv f}\cdot\left(4nc_1\kappa \bb^{-\ve t}\right)^\alpha,
\end{equation}
where $E_{\vv f}=s\,\max_{1\le i\le s} E_{J(x_i)}$. By (\ref{eq2}), $\tilde D^2_{\ppp}$ can be written as a union of disjoint intervals of length
$
\ge \omega/K=(nc_1)^{-1}\kappa\,\bb^{-t(1+\gamma-\ve)}=(nc_1)^{-1}\delta_t.
$
By splitting some of these intervals if necessary, we get a collection $\tilde \cD^2_{\ppp}$ of disjoint intervals $\Delta$ such that
$
\tfrac1{c_1n}\delta_t\le |\Delta|\le \delta_t.
$
Let $$K_0=\max_{\vv f\in\cF_n(I)}\ \max\{4s_{\vv f},(4nc_1)^{1+\alpha}E_{\vv f}\}.$$ Then, by (\ref{eq3}) and the above inequality, we get
\begin{equation}\label{eq4}
\#\tilde \cD^2_{\ppp}\le \frac{E_{\vv f}\cdot\left(4nc_1\kappa \bb^{-\ve t}\right)^\alpha}{\tfrac1{c_1n}\delta_t}\le \frac{K_0\,(\kappa\,\bb^{-\ve t})^{\alpha}}{2\delta_t}.
\end{equation}
Let $\cD^2_{\ppp}$ be the collection of all the intervals in $\tilde \cD^2_{\ppp}$ together with the $2s$ intervals
$[a_i,a_i+\omega/2K]$ and $[b_i-\omega/2K,b_i]$ $(1\le i\le s)$. It is easily seen that $2s$ is less than or equal to the right hand side of (\ref{eq4}). Then, by (\ref{eq4}) and the definition of $\cD^2_{\ppp}$, we get (\ref{eq5}) and (\ref{eq6}). Also, by construction, we see that $\cD^2_{\ppp}$ is a cover of $D^2_{\ppp}$. The proof is thus complete.
\qed

\section{A Cantor sets framework}

Let $R\ge2$ be an integer. Given a collection $\cI$ of compact intervals in $\R$, let $\tfrac1R\cI$ denote the collection of intervals obtained by dividing each interval in $\cI$ into $R$ equal closed subintervals. For example, for $R=3$ and $\cI=\{[0,1]\}$ we have that $\frac1R\cI=\big\{[0,\tfrac13],[\tfrac13,\tfrac23],[\tfrac23,1]\big\}$.
Let $I_0\subset\R$ be a compact interval. The sequence $(\cI_q)_{q\ge0}$ will be called \emph{an $R$-sequence in $I_0$} if
\begin{equation}\label{eq22}
\cI_0=\{I_0\} \qqand\cI_q\subset\tfrac1R\cI_{q-1}\quad\text{for }q\ge1\,.
\end{equation}
The intervals lying in $\cI_q$ will be called to be of {\em level $q$}.
Thus, the intervals of level $q$ are obtained from intervals of level $q-1$ by, firstly, splitting the intervals of $\cI_{q-1}$ into $R$ equal parts to form $\tfrac1R\cI_{q-1}$, and, secondly, removing some of the intervals from  $\tfrac1R\cI_{q-1}$ to form $\cI_q$. Given $q\in\N$, the intervals that are being removed in this procedure will be denoted by
$$
\wcI_q\ \stackrel{\rm def}{=}\ \big(\tfrac1R\cI_{q-1}\big)\setminus \cI_q\,.
$$
Naturally, $I_q$ will denote any interval from the collection $\cI_q$, that is any interval of level $q$. Observe that
\begin{equation}\label{vv5}
|I_q|=R^{-q}|I_0|\qquad\text{for $q\ge0$.}
\end{equation}

By definition, given $I_q\in\cI_q$ with $q\ge1$, there is a unique interval $I_{q-1}\in \cI_{q-1}$ such that $I_{q}\subset I_{q-1}$; this interval $I_{q-1}$ will be called the \emph{precursor} of $I_q$. Obviously it is independent of the choice of the $R$-sequence $(\cI_q)_{q\ge0}$ with $I_q\in\cI_q$.

We also define the limit set of $(\cI_q)_{q\ge0}$ as
\begin{equation}\label{eq+06}
\cKI\ \stackrel{\rm def}{=}\ \bigcap_{q\ge0}\hspace*{1ex} \bigcup_{I_q\in\cI_q}I_q.
\end{equation}
This is \emph{a Cantor type set}.
The classical middle third Cantor set can be constructed this way in an obvious manner with $R=3$ and $I_0=[0,1]$. Theorem~\ref{t2} will be proved by finding suitable Cantor type sets $\cKI$. The construction of the corresponding $R$-sequences will be based on removing the intervals that intersect dangerous intervals -- see \S\ref{dan}.

Note that if $\cI_q\neq\emptyset$ for all $q$ so that $(\cI_q)_{q\ge0}$ is genuinely an infinite sequence, then $\cKI\neq\emptyset$. However, ensuring that $\cKI$ is large requires better understanding of the sets $\cI_q$. There are various techniques in fractal geometry that are geared towards this task -- see \cite{Falconer-03:MR2118797}. We shall use a recent powerful result of Badziahin and Velani \cite{Badziahin-Velani-MAD} restated below using our notation. Naturally, if we expect that the Cantor set $\cKI$ is large, then the number of removed intervals at level $q$, that is the cardinality of $\wcI_q$, should be relatively small. In what follows, given $q\in\N$ and an interval $J$, let
$$
\wcI_q\sqcap J\ \stackrel{\rm def}{=}\ \{I_q\in\wcI_q:I_q\subset J\}\,.
$$
This denotes the subcollection of removed intervals (when going from level $q-1$ to level $q$) that lie over a given interval $J$.
The key characteristic that is `assessing' the proportion of removed intervals at a particular level is given by
\begin{equation}\label{defl}
   d_q(\cI_q) \ =\ \min_{\{\wcI_{q,p}\}}\ \ \sum_{p=0}^{q-1} \ \left(\frac4R\right)^{q-p} \max_{I_p\in\cI_p}\#\big(\wcI_{q,p}\sqcap I_p\big)\ ,
\end{equation}
where the minimum is taken over all partitions $\{\wcI_{q,p}\}_{p=0}^{q-1}$ of $\wcI_q$, that is $\wcI_q=\bigcup_{p=0}^{q-1}\wcI_{q,p}$. Also define the corresponding global characteristic as
$$
d((\cI_q)_{q\ge0})=\sup_{q>0}d_q(\cI_q).
$$
The goal is to ensure that $d((\cI_q)_{q\ge0})$ is small. Then as we shall shortly see the corresponding Cantor set is large. Note that when estimating $d_q(\cI_q)$ the key is to arrange the removed intervals into a partition
$\bigcup_{p=0}^{q-1}\wcI_{q,p}$ which makes the sum on the right of \eqref{defl} small.

\begin{theorem}[Theorem~4 in \cite{Badziahin-Velani-MAD}]\label{t4}
Let $R\ge4$ be an integer, $I_0$ be a compact interval in $\R$ and $(\cI_q)_{q\ge0}$ be an $R$-sequence in $I_0$. If $d((\cI_q)_{q\ge0})\le1$ then
\begin{equation}\label{vb+1}
\dim \cKI\ \ge \ \left(1-\frac{\log2}{\log R}\right).
\end{equation}
\end{theorem}

In order to facilitate the comparison of Theorem~\ref{t4} to \cite[Theorem~4]{Badziahin-Velani-MAD} we summarise the correspondence between the notation and objects used in this paper and in \cite{Badziahin-Velani-MAD}:
$$
\begin{array}{r|l}
\text{Our notation/object} & \text{Corresponding notation/object in \cite{Badziahin-Velani-MAD}}  \\ \hline
q & n+1\\[1ex] \hline
R & R_n\text{ (allowed to vary with $n$)}\\[1ex] \hline
\tfrac1R\cI_{q-1}& \cI_{n+1}\\[1ex] \hline
\cI_q & \cJ_{n+1}\\[1ex] \hline
p & n-k\ (0\le k\le n)\\[1ex] \hline
\max_{I_p\in\cI_p}\#\big(\wcI_{q,p}\sqcap I_p\big) & r_{n-k,n}\\[1ex] \hline
\end{array}
$$
Given the above correspondence table, it is readily verified that our condition $d((\cI_q)_{q\ge0})\le 1$ corresponds to condition (16) within \cite[Theorem~4]{Badziahin-Velani-MAD}. Hence Theorem~\ref{t4} above is an immediate consequence of Theorem~4 from \cite{Badziahin-Velani-MAD}.

\bigskip

Let $M>1$, $X\subset \R$ and $I_0$ be a compact interval. We will say that $X$ is \emph{$M$-Cantor rich in $I_0$} if for any $\ve>0$ and any integer $R\ge M$ there exists an $R$-sequence $(\cI_q)_{q\ge0}$ in $I_0$ such that $\cKI\subset X$ and $d((\cI_q)_{q\ge0})\le \ve$. We will say that $X$ is \emph{Cantor rich in $I_0$} if it is \emph{$M$-Cantor rich in $I_0$} for some $M$. We will say that $X$ is \emph{Cantor rich} if it is \emph{Cantor rich in $I_0$} for some compact interval $I_0$.
The following statement readily follows from Theorem~\ref{t4} and our definitions.

\begin{theorem}\label{t5}
Any Cantor rich set $X$ satisfies \ $\dim X=1$.
\end{theorem}

We now proceed with a discussion of the intersections of Cantor rich sets. To some extent this already appears in \cite[Theorem~5]{Badziahin-Velani-MAD} and  in \cite{Badziahin-Pollington-Velani-Schmidt}. First we prove
the following auxiliary statement.

\begin{lemma}\label{lemma16}
Let $(\cI^j_q)_{q\ge0}$ be a family of $R$-sequences in $I_0$ indexed by $j$. Given $q\in\Z_{\ge0}$, let $\cJ_q=\bigcap_{j}\cI^j_q$. Then $(\cJ_q)_{q\ge0}$ is an $R$-sequence in $I_0$ such that
\begin{equation}\label{eq+02}
\textstyle\wcJ_q\subset \bigcup_{j}\wcI^j_q\qquad\text{for all }q\ge0
\end{equation}
and
\begin{equation}\label{eq+05}
\textstyle\cK((\cJ_q)_{q\ge0})\subset \bigcap_{j}\cK((\cI^j_q)_{q\ge0}).
\end{equation}
\end{lemma}

\begin{proof}
The validity of (\ref{eq22}) for $(\cJ_q)_{q\ge0}$ follows from the uniqueness of the precursor of an interval in any $R$-sequence from that sequence and the fact that $\cI^j_0=\{I_0\}$ for all $j$, which means that $\cJ_0=\bigcap_j\cI^j_0=\{I_0\}$. Thus, $(\cJ_q)_{q\ge0}$ is truly an $R$-sequence. The inclusion (\ref{eq+02}) is obvious for $q=0$ for both sides of the inclusion are empty sets in this case. To see (\ref{eq+02}) for $q>0$,
observe that $\cJ_{q-1}\subset \cI^j_{q-1}$ and this implies that $\tfrac1R\cJ_{q-1}\subset \tfrac1R\cI^j_{q-1}$ for each $j$. Then we have
$$
\begin{array}{rcl}
\wcJ_q \ = \ \tfrac1R\cJ_{q-1}\setminus\cJ_q& =& \tfrac1R\cJ_{q-1}\setminus\bigcap_{j}\cI^j_q \ =\ \bigcup_{j}\big(\tfrac1R\cJ_{q-1}\setminus\cI^j_q\big)\\[2ex]
 & \subset & \bigcup_{j}\big(\tfrac1R\cI^j_{q-1}\setminus\cI^j_q\big) \ = \ \bigcup_{j}\wcI^j_q.
\end{array}
$$
Finally, by the inclusion $\cJ_q\subset\cI^j_q$, we have that $\bigcup J_q\subset \bigcup I^j_q$ for each pair of $j$ and $q$, where the union is taken over $J_q\in\cJ_q$ and $I^j_q\in\cI^j_q$ respectively. Hence, by (\ref{eq+06}), we have that $\cK((\cJ_q)_{q\ge0})\subset\cK((\cI^j_q)_{q\ge0})$ for all $j$, whence (\ref{eq+05}) now follows.
\end{proof}

\begin{theorem}\label{t6}
Let $I_0$ be a compact interval. Then any countable intersection of $M$-Cantor rich sets in $I_0$ is $M$-Cantor rich in $I_0$. In particular, any finite intersection of Cantor rich sets in $I_0$ is Cantor rich in $I_0$.
\end{theorem}

\begin{proof}
Let $\{X_j\}_{j\in\N}$ be a collection of $M$-Cantor rich sets in $I_0$. Let $\ve>0$. Then, by definition, for each $j\in\N$ and $R\ge M$ there is an $R$-sequence $(\cI^j_q)_{q\ge0}$ in $I_0$ such that $\cK((\cI^j_q)_{q\ge0})\subset X_j$ and $d_q(\cI^j_q)\le \ve 2^{-j}$ for all $q>0$. By \eqref{defl}, for each $j$ and $q>0$ there exists a partition $\{\wcI^j_{q,p}\}_{p=0}^{q-1}$ of $\wcI^j_q$ such that
\begin{equation}\label{vb+100}
\sum_{p=0}^{q-1} \ \left(\frac4R\right)^{q-p} \max_{I_p\in\cI^j_p}\#\big(\wcI^j_{q,p}\sqcap I_p\big)\ \le \ve 2^{-j}.
\end{equation}
For $q\in\Z_{\ge0}$ define $\cJ_q=\bigcap_{j\in\N}\cI^j_q$ and $\wcJ_{q,p} =  \wcJ_q\cap\bigcup_{j\in\N} \wcI^j_{q,p}$. Since $\wcI^j_q=\bigcup_{p=0}^{q-1}\wcI^j_{q,p}$ for each $j$, by \eqref{eq+02}, we have that $\wcJ_q=\bigcup_{p=0}^{q-1}\wcJ_{q,p}$, where $q>0$.
Then, for each $q>0$ we get that
$$
\sum_{p=0}^{q-1} \left(\frac4R\right)^{q-p} \!\!\max_{J_p\in\cJ_p}\#\big(\wcJ_{q,p}\sqcap J_p\big)
\le
\sum_{j=1}^\infty\sum_{p=0}^{q-1} \left(\frac4R\right)^{q-p} \!\!\max_{I_p\in\cI^j_p}\#\big(\wcI^j_{q,p}\sqcap I_p\big).
$$
This inequality together with \eqref{vb+100} and the definition of $d((\cJ_q)_{q\ge0})$ implies that $d((\cJ_q)_{q\ge0})\le\ve$. By \eqref{eq+05} and the fact that $\cK((\cI^j_q)_{q\ge0})\subset X_j$ for each $j$, we have that $\cK((\cJ_q)_{q\ge0})\subset \bigcap_jX_j$. Thus the intersection $\bigcap_jX_j$ meets the definition of $M$-Cantor rich sets and the proof is complete.
\end{proof}

\bigskip

The winning sets in the sense of Schmidt have been used a lot to investigate various sets of badly approximable points. Hence we suggest the following

\medskip

\noindent\textbf{Problem 3:} Verify if an $\alpha$-winning set in $\R$ as defined by Schmidt \cite{Schmidt-1980} is $M$-Cantor rich for some $M$ and, if this so, find an explicit relation between  $M$ and $\alpha$.

\section{Proof of Theorem~2}

The following proposition is a key step to establishing Theorem~\ref{t2}. We will use the Vinogradov symbol $\ll$ to simplify the calculations. The expression $X\ll Y$ will mean that $X\le CY$ for some $C>0$, which only depends on $n$, the family of maps $\cF_n(I)$ from Theorem~\ref{t2} and the interval $I_0$ occurring in Property~F.

\begin{proposition}\label{prop2}
Let $\cF_n(I)$ be as in Theorem~\ref{t2}, $I_0\subset I$ be a compact interval satisfying Property~F, $c_0,c_1$ be the same as \eqref{eee05}, $\sigma=1-(2n)^{-4}$ and $\kappa_0$ be as in Proposition~\ref{prop1}. Further, let
\begin{equation}\label{varrho}
  \varrho_1 = nc_1|I_0|+1\qquad\text{and}\qquad \varrho_0 = \varrho_1|I_0|+1
\end{equation}
and let
\begin{equation}\label{R_0}
R_0 = \max\left\{\varrho_0,\ nc_1,\ \frac{2^{n+1}\varrho_0\varrho_1(n+1)!}{c_0}\right\}
\end{equation}
and
\begin{equation}\label{m_0}
m_0 = \max\left\{4,\frac{-\log \kappa_0}{\log R_0}+1\right\}\,.
\end{equation}
Then for any $\vv f\in\cF_n(I)$, $\rr\in\cR_n$ and any integers $m\ge m_0$ and $R\ge R_0$, there exists an $R$-sequence $(\cI_q)_{q\ge0}$ in $I_0$ such that
\begin{enumerate}
  \item[\rm(i)] for any $t\in\N$ and any $I_{t+m}\in \cI_{t+m}$ we have that
\begin{equation}\label{vb8vb8}
  \delta\big(\gt \Gx\Z^{n+1}\big)\ge1\qquad\text{for all }x\in I_{t+m};
\end{equation}
   where $\gt=\gtr$ is given by \eqref{gtr} with $\bb^{1+\gamma}=R$, $\gamma=\gamma(\rr)$ and\\ $\Gx=G(\kappa;\vv f(x))$ is given by \eqref{eq1} with $\kappa=R^{-m};$
  \item[\rm(ii)] if ~$q\le m$~ then ~$\#\wcI_q=0;$
  \item[\rm(iii)] if~ $q=t+m$ for some $t\in\N$ then $\wcI_{q}$ can be written as the union $\wcI_q=\bigcup_{p=0}^{q-1}\wcI_{q,p}$ such that for integers $p=t+3-2\ell$ with $0\le\ell\le \ell_t=[t/2n]+1$ and $I_{p}\in \cI_{p}$  we have that
  \begin{equation}\label{eq21}
\displaystyle\#(\wcI_{q,p}\sqcap I_{p})\ \ll \ R^{\frac{1+\lambda}{2}(q-p)-\frac{1-\lambda}{2}m+3}\,,
\end{equation}
    \begin{equation}\label{vv1}
\#\wcI_{q,0}\ll R^{\sigma q}
\end{equation}
and $\wcI_{q,p}=\emptyset$ for all other $p<q$, where $\lambda=\lambda(\rr)$ is given by \eqref{eq12}.
\end{enumerate}
\end{proposition}

\begin{proof}
Note that since $\varrho_0,\varrho_1>1$ and $c_0<1$, we have that $R_0>4$. Let $m \ge m_0$ and $R\ge R_0$ be any integers. Define $\cI_0=\{I_0\}$ and then for $q=1,\dots,m$ let $\cI_q=\tfrac1R\cI_{q-1}$. In this case conditions (i) and (iii) are irrelevant, while (ii) is obvious. Continuing by induction, let $q=t+m$ with $t\ge1$ and let us assume that $\cI_{q'}$ with $q'<q$ are given and satisfy conditions (i)--(iii). Define $\cI_q$ to be the collection of intervals from $\tfrac1R\cI_{q-1}$ that satisfy (\ref{vb8vb8}). By construction, (i) holds, (ii) is irrelevant and we only need to verify condition (iii). We shall assume that $\wcI_q\neq\emptyset$ as otherwise (iii) is obvious.
By construction, $\wcI_q$ consists of intervals $I_q$ such that $\delta(\gt \Gx\Z^{n+1})<1$ for some $x\in I_q$. Recall that this is equivalent to the existence of $(a_0,\vv a)\in\Z^{n+1}$ with $\vv a\neq\vv0$ satisfying the system (\ref{eq+08}). We shall use Propositions~\ref{prop0} and \ref{prop1} and Lemma~\ref{lemma12} to estimate the number of these intervals $I_q$.
Before we proceed with the estimates note that, by (\ref{eee05}) and (\ref{eee18}), the validity of (\ref{eq+08}) implies that $|\vv a.\vv f'(x)|\le nc_1 \max_{1\le j\le n}|a_j|\le nc_1\max_{1\le j\le n}\bb^{r_jt}=nc_1\bb^{\gamma t}$. Thus,
\begin{equation}\label{eq+16}
\forall\ x\in I_0\qquad \delta(\gt \Gx\Z^{n+1})<1\quad\Rightarrow\quad
|\vv a.\vv f'(x)|\le nc_1\bb^{\gamma t}.
\end{equation}
The arguments split into two cases depending on the size of $t$ as follows. Note that in view of our choice of $m_0$ we have that
$$
\kappa=R^{-m}<\kappa_0
$$
and so Proposition~\ref{prop1} is applicable as appropriate.

\bigskip

\noindent\textbf{Case 1: $t\le 2nm$}. In this case let $\wcI_{q,0}=\wcI_q$ and $\wcI_{q,p}=\emptyset$ for $0<p<q$. Then, the only thing we need to verify is (\ref{vv1}). Let $\ve=0$. Then, by (\ref{eq+16}), we have that
\begin{equation}\label{vv3}
\big\{x\in I_0:\delta(\gt \Gx\Z^{n+1})<1\big\}=D^2_{0},
\end{equation}
where $D^2_{0}=D^2_{t,\ve,\rr,b,\kappa,\vv f}$ (with $\ve=0$) as defined in Proposition~\ref{prop1}.
Hence, $\#\wcI_{q,0}$ is bounded by the number of intervals in $\frac1R\cI_{q-1}$ that intersect an interval from the corresponding collection $\cD^2_{0}$ of intervals arising from Proposition~\ref{prop1}. By (\ref{vv5}), the intervals in $\frac1R\cI_{q-1}$ are of length $R^{-q}|I_0|$. By (\ref{eq5}), the intervals from $\cD^2_{0}$ have length $\le\delta_t=\kappa\,\bb^{-t(1+\gamma)}$. Hence, each interval from $\cD^2_{0}$ can intersect at most
$\delta_t/(R^{-q}|I_0|)+2=\kappa\,\bb^{-t(1+\gamma)}R^{q}|I_0|^{-1}+2$ intervals from $\tfrac1R\cI_{q-1}$. Since $\bb^{1+\gamma}=R$, $\kappa=R^{-m}$ and $q=t+m$, we have that $\delta_t/(R^{-q}|I_0|)= |I_0|^{-1}$. Hence each interval from $\cD^2_0$ can intersect $\ll \delta_tR^{q}$
intervals from $\tfrac1R\cI_{q-1}$. Then, by (\ref{eq6}), we get
\begin{equation}\label{eq+04x}
\#\wcI_{q,0}\ \ll \delta_tR^{q}\times \frac{K_0\,\kappa^{\alpha}}{\delta_t}\ll R^q\times \kappa^{\alpha}.
\end{equation}
Using $q=t+m$, $\kappa=R^{-m}$ and $t\le2nm$ we obtain from (\ref{eq+04x}) that
\begin{equation}\label{eq+04}
\#\wcI_{q,0}\ \ll R^{t+m}\times (R^{-m})^{\alpha}\le R^{\left(1-\frac{\alpha}{2n+1}\right)(t+m)} = R^{\left(1-\frac{\alpha}{2n+1}\right)q}.
\end{equation}
Recall from Proposition~\ref{prop1} that $\alpha=\frac1{(n+1)(2n-1)}$. Consequently, $\sigma\ge 1-\frac{\alpha}{2n+1}$ and (\ref{eq+04}) implies (\ref{vv1}).

\bigskip

\noindent\textbf{Case 2: $t>2nm$}. Let $\ve=(2n)^{-1}$. Since $\sum_ir_i=1$ and $\gamma=\max\{r_1,\dots,r_n\}$, we have that $\gamma\ge1/n$. Hence $\ve<\gamma$. Recall that $R>nc_1$. Then, by (\ref{eq+16}) and the choice of $\ve$, for any $x\in I_0$ such that $\delta(\gt \Gx\Z^{n+1})<1$ we have that either
$$
|\vv a.\vv f'(x)| < nc_1\bb^{(\gamma-\ve) t}
$$
or for some $\ell\in\Z$ with $0\le\ell\le \ell_t= [t/2n]+1$
$$
 \bb^{\gamma t-(1+\gamma)\ell}\le~ |\vv a.\vv f'(x)| < \bb^{\gamma t-(1+\gamma)(\ell-1)}.
$$
Then, once again using the equivalence of $\delta(\gt \Gx\Z^{n+1})<1$ to the existence of $(a_0,\vv a)\in\Z^{n+1}$ with $\vv a\neq\vv0$ satisfying (\ref{eq+08}), we write that
\begin{equation}\label{vv3+}
\big\{x\in I_0:\delta(\gt \Gx\Z^{n+1})<1\big\}=\bigcup_{\ell=0}^{\ell_t}\ \ \bigcup_{\vv a\in\Z^n\setminus\{\vv0\}}\ \ \bigcup_{a_0\in\Z}\ D^1_{\ell}(a_0,\vv a)\cup D^2_{\ppp},
\end{equation}
where
$$
D^1_{\ell}(a_0,\vv a)=D^1_{t,\ell,\rr,\bb,\kappa,\vv f}(a_0,\vv a) \qqand D^2_{\ppp}=D^2_{t,\ve,\rr,b,\kappa,\vv f}
$$
as defined in Propositions~\ref{prop0} and \ref{prop1} respectively.

By definition, intervals in $\wcI_q$ are characterised by having a non-empty intersection with the left hand side of (\ref{vv3+}). We now use the right hand  side of (\ref{vv3+}) to define the subcollections $\wcI_{q,p}$ of $\wcI_q$. More precisely, for $p=t+3-2\ell$ with $0\le\ell\le \ell_t$ let $\wcI_{q,p}$ consist of the intervals $I_q\in\wcI_q$ that intersect $D^1_{\ell}(a_0,\vv a)$ for some $\vv a\in\Z^n\setminus\{\vv0\}$ and $a_0\in\Z$. Next, let $\wcI_{q,0}$ consist of the intervals $I_q\in\wcI_q$ that intersect $D^2_{\ppp}$. Finally, define $\wcI_{q,p}=\emptyset$ for all other $p<q$. By (\ref{vv3+}), it is easily seen that $\wcI_q=\bigcup_{p=0}^{q-1}\wcI_{q,p}$. It remains to verify (\ref{eq21}) and (\ref{vv1}).

\bigskip

\noindent$\bullet$ \textit{Verifying (\ref{vv1}).}
This is very much in line with Case 1. The goal is to count the number intervals in $\frac1R\cI_{q-1}$ that intersect some interval from the collection $\cD^2_{\ppp}$ arising from Proposition~\ref{prop1}. By (\ref{vv5}), the intervals in $\frac1R\cI_{q-1}$ are of length $R^{-q}|I_0|$. By (\ref{eq5}), the intervals from $\cD^2_{\ppp}$ have length $\le\delta_t=\kappa\,\bb^{-t(1+\gamma-\ve)}$. Hence, each interval from $\cD^2_{\ppp}$ can intersect at most
$\delta_t/(R^{-q}|I_0|)+2\ll \delta_tR^{q}$ intervals from $\tfrac1R\cI_{q-1}$. Then, by (\ref{eq6}), we get
$$
\#\wcI_{q,0}\ \ll \delta_tR^{q}\times \frac{K_0\,(\kappa\,\bb^{-\ve t})^{\alpha}}{\delta_t}\ll R^q\times (\kappa\,\bb^{-\ve t})^{\alpha}.
$$
Using $\kappa=R^{-m}$, $\bb^{1+\gamma}=R$, $q=t+m$ and $0<\gamma\le 1$, we obtain that
\begin{equation}\label{eq+04+}
\#\wcI_{q,0}\ \ll R^{q}\times (R^{-m}\,R^{-\ve t/(1+\gamma)})^{\alpha}\le R^{q}R^{-\frac{\ve\alpha}{2}(t+m)}
= R^{(1-\ve\alpha/2)q} .
\end{equation}
Once again using the value of $\alpha$ from Proposition~\ref{prop1} we verify that $\sigma\ge 1-\frac12\ve\alpha$ and so (\ref{eq+04+}) implies (\ref{vv1}) as required.

\bigskip

\noindent$\bullet$ \textit{Verifying (\ref{eq21}).}
Let $p=t+3-2\ell$ with $0\le\ell\le \ell_t$ and $I_{p}\in \cI_{p}$. Let $S(I_p)$ be the set of points $(a_0,\vv a)\in\Z^{n+1}$ with $\vv a\neq\vv0$ such that $D^1_{\ell}(a_0,\vv a)\cap I_{p}\not=\emptyset$.
By Proposition~\ref{prop0}, for every $(a_0,\vv a)\in S(I_p)$ any interval in $\cD^1_{\ell}(a_0,\vv a)$ is of length
$$
\le \kappa \bb^{-(1+\gamma)(t-\ell)}=R^{-m}R^{-(t-\ell)}=R^{-(t+m-\ell)}=R^{\ell-q}.
$$
as $\kappa=R^{-m}$, $\bb^{1+\gamma}=R$ and $q=t+m$. Then, by (\ref{vv5}), any interval from $\cD^1_{\ell}(a_0,\vv a)$ intersects $\ll R^{\ell}$ intervals from $\tfrac1R\cI_{q-1}$. By Proposition~\ref{prop0}, $\#\cD^1_{\ell}(a_0,\vv a)\,\ll\, 1$.
Hence,
\begin{equation}\label{eq+17x}
\#(\wcI_{q,p}\sqcap I_{p})\, \ll\,  \#S(I_p)\times R^\ell
\end{equation}
and our main concern becomes to obtain a bound for $\#S(I_p)$. We shall prove that
\begin{equation}\label{eq+07}
\#S(I_p)\,\ll\, R^{\frac{\tau}{1+\gamma}+\lambda(m+\ell-1)}.
\end{equation}
Armed with this estimate establishing (\ref{eq21}) and thus completing our task becomes simple. Indeed,
using (\ref{eq+17x}) and (\ref{eq+07}) gives
$$
\#(\wcI_{q,p}\sqcap I_{p})\, \ll\,
R^{\frac{1+\lambda}2(2\ell+m-3)-\frac{1-\lambda}{2}m+\frac{3+\lambda}{2}+\frac\tau{1+\gamma}},
$$
which implies (\ref{eq21}) upon observing that $2\ell+m-3=q-p$ and $\frac{3+\lambda}{2}+\frac\tau{1+\gamma}<3$.

\bigskip
\bigskip

\noindent\textit{Proof of }(\ref{eq+07}).
We assume that $S(I_p)\neq\emptyset$ as otherwise (\ref{eq+07}) is trivial. The proof will be split into several relatively simple steps.

\bigskip

\noindent \underline{Step 1}\,: We show that for any $(a_0,\vv a)\in S(I_{p})$ and any $x\in I_p$ we have
\begin{equation}\label{eq44}
|a_0+\vv a.\vv f(x)|
\ < \ \varrho_0\bb^{-t+(1+\gamma)(\ell-2)}
\qquad\text{and}\qquad
|\vv a.\vv f'(x)|< \varrho_1 \bb^{\gamma t-(1+\gamma)(\ell-1)},
\end{equation}
where $\varrho_0$ and $\varrho_1$ are given by \eqref{varrho}.

\medskip

First we prove the right hand side of \eqref{eq44}. To this end, fix any $(a_0,\vv a)\in S(I_{p})$ and let $x_0\in D^1_{\ell}(a_0,\vv a)\cap I_{p}$. To simplify notation define $f(x)=a_0+\vv a.\vv f(x)$.
By the Mean Value Theorem, for any $x\in I_{p}$ we have
\begin{equation}\label{vv4}
|f'(x)| =  |f'(x_0)+f''(\tilde x_0)(x-x_0)|\ \le \ |f'(x_0)|+|f''(\tilde x_0)(x-x_0)|,
\end{equation}
where $\tilde x_0$ is a point between $x$ and $x_0$. By the definition of $D^1_{\ell}(a_0,\vv a)$, we have that $|f'(x_0)|\le \bb^{\gamma t-(1+\gamma)(\ell-1)}$.
Proceeding as in \eqref{eq+14},
we get that $|f''(\tilde x_0)|< nc_1\bb^{\gamma t}$. Substituting the estimates for $|f'(x_0)|$ and $|f''(\tilde x_0)|$ into (\ref{vv4}) and using the inequity
\begin{equation}\label{vvnn}
|x-x_0|\le |I_{p}|= R^{-p}|I_0|=R^{-(t+3-2\ell)}|I_0|
\end{equation}
implied by (\ref{vv5}), we get
$$
\begin{array}{rclcl}
|f'(x)|&  < & \bb^{\gamma t-(1+\gamma)(\ell-1)}+nc_1 \bb^{\gamma t}\times R^{-(t+3-2\ell)}|I_0|.
\end{array}
$$
Since $\bb^{1+\gamma}=R$, we have that
\begin{equation}\label{eq+09}
|f'(x)|<  \bb^{\gamma t-(1+\gamma)(\ell-1)}+nc_1|I_0|\bb^{\gamma t-(1+\gamma)(t+3-2\ell)}.
\end{equation}
Since $\ell\le \ell_t<t/4+1$, one easily verifies that $(t+3-2\ell)>(\ell-1)$. Therefore (\ref{eq+09}) implies the right hand side of (\ref{eq44}).

\medskip

Now we prove the left hand side of \eqref{eq44}. Again fix any $(a_0,\vv a)\in S(I_{p})$, $x_0\in D^1_{\ell}(a_0,\vv a)\cap I_{p}$ and let $f(x)=a_0+\vv a.\vv f(x)$.
By the Mean Value Theorem, for any $x\in I_{p}$ we have that
\begin{equation}\label{vv4nn}
|f(x)| =  |f(x_0)+f'(\widehat x_0)(x-x_0)|\ \le \ |f(x_0)|+|f'(\widehat x_0)(x-x_0)|,
\end{equation}
where $\widehat x_0$ is a point between $x$ and $x_0$. In particular, $\widehat x_0\in I_p$ and therefore, by the right hand side of \eqref{eq44}, which we have already established, $|f'(\widehat x_0)|< \varrho_1\bb^{\gamma t-(1+\gamma)(\ell-1)}$. By the definition of $D^1_{\ell}(a_0,\vv a)$, we have that $|f(x_0)|\le \kappa b^{-t}=b^{-t-m(1+\gamma)}$. Hence, using these estimates together with inequality \eqref{vvnn} and equation $\bb^{1+\gamma}=R$, we get from \eqref{vv4nn} that
\begin{align}
\nonumber |f(x)|&~<\,  b^{-t-m(1+\gamma)}+\varrho_1\bb^{\gamma t-(1+\gamma)(\ell-1)}\times b^{-(1+\gamma)(t+3-2\ell)}|I_0|\\[0.5ex]
&~=\,b^{-t-m(1+\gamma)}+\varrho_1|I_0|\bb^{-t+(1+\gamma)(\ell-2)}\,.\label{vvnn2}
\end{align}
Since $m\ge m_0\ge 4$ we have that $-m(1+\gamma)\le (1+\gamma)(\ell-2)$ for all $\ell\ge0$. Therefore (\ref{vvnn2}) implies the left hand side of (\ref{eq44}).

\bigskip
\bigskip

\noindent \underline{Step 2}\,: Now we utilize (\ref{eq44}) to show that $\rank S(I_p)\le n-z$. First of all,
observe that if $(a_0,\vv a)\in S(I_{p})$, where $\vv a=(a_1,\dots,a_n)$, then $|a_j|<\bb^{r_jt}=1$ whenever $r_j=0$. Since $a_j\in\Z$ in this case, we have that
\begin{equation}\label{eq+15}
\forall\ (a_0,\vv a)\in S(I_{p})\qquad a_j=0\qquad \text{whenever}\quad r_j=0.
\end{equation}
Let $J=\{j:r_j\neq0\}$ and $\overline J=\{1,\dots,n\}\setminus J$. Note that $J$ contains exactly $n-z>0$ elements, where $z=z(\rr)$ is the number of zeros in $\rr$. Let $J_0$ be the subset of $J$ obtained by removing the smallest index $j_0$ such that $r_{j_0}=\gamma(\rr)$. Note that if $\rr$ has only one non-zero component then $J_0=\emptyset$. Let $x\in I_p$. Then, using (\ref{eq+15}) and (\ref{eq44}) we obtain that every $(a_0,\vv a)\in S(I_{p})$ satisfies the system
\begin{equation}\label{eq+10}
\left\{
\begin{array}{rcl}
|a_0+\sum_{j\in J} a_j f_j(x)| &<& \varrho_0\bb^{-t+(1+\gamma)(\ell-2)},   \\[0.5ex]
|\sum_{j\in J} a_j f'_j(x)| &<& \varrho_1\bb^{\gamma t-(1+\gamma)(\ell-1)},\\[0.5ex]
|a_j|&<& \bb^{r_jt} \qquad (j\in J_0),\\[0.5ex]
a_j & = & 0\qquad\quad (j\in \overline J)\,,
\end{array}
\right.
\end{equation}
where $\varrho_0$ and $\varrho_1$ ar given by \eqref{varrho}.
Let $\vv B_{p,x}$ denote the set of $(a_0,a_{1},\dots,a_{n})\in\R^{n+1}$ satisfying (\ref{eq+10}). Then, $S(I_p)\subset \vv B_{p,x}$. Clearly, $\vv B_{p,x}$ is a convex body lying over the $n-z+1$ dimensional linear subspace of $\R^{n+1}$ given by the equations $a_j=0$ for $j\in \overline J$. As is well known the $n-z+1$-dimensional volume of $\vv B_{p,x}$ is equal to
$$
\frac{2\varrho_0\bb^{-t+(1+\gamma)(\ell-2)}\times 2\varrho_1\bb^{\gamma t-(1+\gamma)(\ell-1)}\times\prod_{j\in J_0}2b^{r_jt}}{|\Omega|}~=~\frac{2^{n+1-z}\varrho_0\varrho_1b^{-(1+\gamma)}}{|\Omega|}\,,
$$
where $\Omega$ is the determinant of the system of linear forms in the variables $a_j$, $j\in J\cup\{0\}$, staying in the first three lines of \eqref{eq+10}.
Note that $|\Omega|=|f'_{j_0}(x)|$. Hence, using (\ref{eee05}) and the fact that $b^{1+\gamma}=R\ge R_0$, we conclude that the volume of $\vv B_{p,x}$ is
$$
\frac{2^{n+1-z}\varrho_0\varrho_1b^{-(1+\gamma)}}{|f'_{j_0}(x)|}<
\frac{2^{n+1-z}\varrho_0\varrho_1}{c_0R}\le \frac{2^{n+1}\varrho_0\varrho_1}{c_0R_0}\le\frac{1}{(n+1)!}\le\frac{1}{(n-z+1)!}\,.
$$
In this case, Lemma~\ref{lemma09} is applicable and we have that $\rank S(I_p)\le n-z$ as claimed at the start of Step 2.

\bigskip
\bigskip

\noindent \underline{Step 3}\,: Finally, we obtain (\ref{eq+07}). To this end,
let $\Gamma$ denote the $\Z$-span of $S(I_p)$. Since $\rank S(I_p)\le n-z$, we have that $\rank\Gamma\le n-z$.
Discarding the second inequality from (\ref{eq+10}) and using the fact that $\varrho_0\le R=b^{1+\gamma}$, which is implied by \eqref{R_0}, we obtain that the points $(a_0,\vv a)\in S(I_p)$ satisfy the system
\begin{equation}\label{eq+13}
\left\{
\begin{array}{rcl}
|a_0+\sum_{j=1}^n a_j f_j(x)| &<& \bb^{-t+(1+\gamma)(\ell-1)},   \\[0.5ex]
|a_j|&<& \bb^{r_jt} \qquad\qquad\qquad (1\le j\le n)\,.
\end{array}
\right.
\end{equation}
On applying $g^t$ to both sides of the system and dividing its first inequality by $\kappa=R^{-m}=\bb^{-m(1+\gamma)}$, \eqref{eq+13} becomes
$$
\left\{
\begin{array}{rcl}
\bb^t\kappa^{-1}|a_0+\sum_{j=1}^n a_j f_j(x)| &<& \kappa^{-1}\bb^{(1+\gamma)(\ell-1)}=\bb^{(1+\gamma)(m+\ell-1)},   \\[0.5ex]
\bb^{-r_jt}|a_j|&<& 1 \qquad (1\le j\le n)\,.
\end{array}
\right.
$$
Hence, in view of the definitions of $\Pi(\bb,u)$, $\gt=\gtr$ and $G_x=G(\kappa;\vv f(x))$ given in \S\ref{counting}, namely \eqref{eq1}, \eqref{gtr} and \eqref{eq92}, we obtain that
\begin{equation}\label{eq+11}
\gt \Gx S(I_p)\subset\gt \Gx\Gamma\cap\Pi(\bb,u)\qquad\text{ with }\quad u=(1+\gamma)(m+\ell-1).
\end{equation}
Note that $0\le\tau(\rr)\le1/n\le \gamma(\rr)\le1$. Therefore $\lambda(1+\gamma)=(1+\gamma)/(1+\tau)\in [1,2]$. Hence $(m+\ell-1)\le \lambda u\le 2(m+\ell-1)$.
Since $t>2nm$, $m\ge4$ and $\ell-1\le t/2n$, one can easily see that $1<\lambda u<t$. Hence $1\le t-[\lambda u]<t$. Take $x\in I_{t-[\lambda u]}\cap I_p$ for an appropriate interval $I_{t-[\lambda u]}$. By induction, (\ref{vb8vb8}) holds when $t$ is replaced by $t-[\lambda u]$. This verifies (\ref{eee13}) with $\Lambda=\Gx\Gamma$. Clearly, $\rank\Lambda=\rank\Gamma\le n-z$. Hence, by Lemma~\ref{lemma12} and (\ref{eq+11}), we obtain that
$
\#S(I_p)\,=\,\#\gt \Gx S(I_p)\,\ll\, \bb^{\tau}\bb^{\lambda u}.
$
Now (\ref{eq+07}) readily follows upon substituting $\bb=R^{1/(1+\gamma)}$ and $u=(1+\gamma)(m+\ell-1)$.
\end{proof}

\bigskip

The following key statement is essentially a corollary of Proposition~\ref{prop2}.

\begin{theorem}\label{t7}
Let $\cF_n(I)$ be as in Theorem~\ref{t2}, $I_0\subset I$ be a compact interval satisfying Property~F. Then there is a constant $M_0\ge4$ such that for any $\rr\in\cR_n$ and any $\vv f\in\cF_n(I)$ the set\/ $\vv f^{-1}(\Badr)$ is $M$-Cantor rich in $I_0$ for any $M> \max\big\{M_0,16^{1+1/\tau}\big\}$, where $\tau=\tau(\rr)$ is defined by  \eqref{eee03}.
\end{theorem}

\begin{proof}
Let $R_0$ and $m_0$ be as in Proposition~\ref{prop2} and $M_0=\max\{R_0,4^{(2n)^4}\}$. Let $M> \max\big\{M_0,16^{1+1/\tau}\big\}$, $R\ge M$ and $m\ge m_0$. Take any $\vv f\in\cF_n(I)$ and $\vv r\in\cR_n$. Let $(\cI_q)_{q\ge0}$ denote the $R$-sequence in $I_0$ that arises from Proposition~\ref{prop2}. By (\ref{vb8vb8}) and Lemma~\ref{lemma08}, we have that $\cKI\subset\vv f^{-1}(\Badr)$. Thus, by definition, the fact that $\vv f^{-1}(\Badr)$ is $M$-Cantor rich in $I_0$ will follow on showing that $d((\cI_q)_{q\ge0})$ can be made $\le\ve$ for any $\ve>0$.

Observe that $(1-\lambda)/2=\tau/(2+2\tau)$. Then, since $R\ge M>16^{1+1/\tau}$, we have that $4R^{-\frac{1-\lambda}{2}}<1$. By conditions (ii) and (iii) of Proposition~\ref{prop2}, for $q>0$
$$
   \sum_{p=1}^{q-1} \ \left(\frac4R\right)^{q-p} \max_{I_p\in\cI_p}\#\big(\wcI_{q,p}\sqcap I_p\big)\
   \ll
   \sum_{q-p\ge m-3} \left(\frac4R\right)^{q-p} R^{\frac{1+\lambda}{2}(q-p)-\frac{1-\lambda}{2}m+3} <
$$
\begin{equation}\label{vb+200}
<   R^{-\frac{1-\lambda}{2}m+3}\sum_{\ell\ge m-3} \left(4R^{-\frac{1-\lambda}{2}}\right)^{\ell}
=   R^{-\frac{1-\lambda}{2}m+3}\frac{\left(4R^{-\frac{1-\lambda}{2}}\right)^{m-3}}{1-4R^{-\frac{1-\lambda}{2}}}\to0
\end{equation}
as $m\to\infty$.
Further, since $R\ge M> M_0\ge4^{(2n)^4}$, we have that $4R^{-(1-\sigma)}=4R^{-(2n)^{-4}}<1$. Once again, by conditions (ii) and (iii) of Proposition~\ref{prop2}, for $q\le m$ we have that $\#\big(\wcI_{q,0}\sqcap I_0\big)=0$, while for $q>m$
\begin{equation}\label{vb+300}
   \left(\frac4R\right)^{q} \#\big(\wcI_{q,0}\sqcap I_0\big)\ll
\left(\frac4R\right)^{q} R^{\sigma q} = \left(4R^{-(1-\sigma)}\right)^q<\left(4R^{-(2n)^{-4}}\right)^m\to 0
\end{equation}
as $m\to\infty$. By \eqref{defl}, combining \eqref{vb+200} and \eqref{vb+300} gives $d_q(\cI_q)\le \ve$ for all $q>0$ provided that $m$ is sufficiently large. This completes the proof.
\end{proof}

\bigskip
\bigskip

\noindent\textit{Proof of Theorem~\ref{t2}.} Let $I_0$ and $M_0$ be the same as in Theorem~\ref{t7} and $M=\max\big\{M_0,16^{1+1/\tau_0}\big\}+1$, where $\tau_0=\inf\{\tau(\rr):\rr\in W\}$. By \eqref{eq+01}, $\tau_0>0$ and so $M<\infty$. By Theorem~\ref{t7}, $\vv f^{-1}(\Bad(\rr))$ is $M$-Cantor rich in $I_0$ for each $\vv f\in\cF_n(I)$ and each $\rr\in W$. By Theorem~\ref{t6}, so is $S=\bigcap_{\vv f\in\cF_n(I)}\bigcap_{\rr\in W}\vv f^{-1}(\Bad(\rr))$. By Theorem~\ref{t5}, $\dim S=1$. The proof is thus complete.
\hspace*{\fill}\raisebox{-1ex}{$\boxtimes$}

\section{Final remarks}

In this section we discuss possible generalisations of our main results and further problems. First of all, the analyticity assumption within Theorem~\ref{t1} can be relaxed by making use of more general fibering techniques such as that of \cite{Pyartli-1969}. This however leaves the question of whether Theorem~\ref{t1} holds for arbitrary nondegenerate submanifold of $\R^n$ as defined in \cite{Kleinbock-Margulis-98:MR1652916} open.
Beyond nondegenerate manifolds, it would be interesting to obtain generalisations of Theorems~\ref{t1} and \ref{t2} for friendly measures as defined in \cite{Kleinbock-Lindenstrauss-Weiss-04:MR2134453} as well as for affine subspaces of $\R^n$ and their submanifolds -- see \cite{Kleinbock-03:MR1982150} for a related context.
In another direction, it would be interesting to develop the theory of badly approximable systems of linear forms. Removing condition \eqref{eq+01} is another appealing problem that would be settled if the sets of interest were shown to be winning in the sense of Schmidt (see \cite{Schmidt-1980}, \cite{ABV}, \cite[\S1.3]{Badziahin-Velani-Dav} and \cite{An-12}). However, the techniques of this paper could also help accomplishing this task: the key is to make the lower bound on $M$ appearing in Theorem~\ref{t7} independent of $\tau(\rr)$.  Finally, all of the above questions make sense and are of course interesting in the case of Diophantine approximation over $\Q_p$ and in positive characteristic.

\setcounter{section}{0}

\renewcommand{\thesection}{Appendix \Alph{section}}

\section{Proof of Lemma~\ref{dual}}\label{A}

As was mentioned, the equivalence of (i) and (ii) is straightforward and thus left to the reader. The proof of the equivalence of (ii) and (iii) will make use of the following

\begin{lemma}[Mahler \cite{Mahler-39:MR0001241}]\label{lemma06}
Let $L_0,\dots,L_n$ be a system of linear forms in variables $u_0,\dots,u_n$ with real coefficients and determinant $d\neq0$, and let $L'_0,\dots,L'_n$ be the transposed system of linear forms in variables $v_0,\dots,v_n$, so that $\sum_{i=0}^n L_iL'_i=\sum_{i=0}^n u_iv_i$. Let $\lambda=T_0\cdots T_n/|d|$. Suppose there exists an integer point $(u_0,\dots,u_n)\neq(0,\dots,0)$ such that
\begin{equation}\label{eee08}
|L_i(u_0,\dots,u_n)|\le T_i\qquad (0\le i\le n).
\end{equation}
Then there exists an integer point $(v_0,\dots,v_n)\neq(0,\dots,0)$ such that
\begin{equation}\label{eee09}
|L'_0(v_0,\dots,v_n)|\le n\lambda/T_0~ \qand ~|L'_i(v_0,\dots,v_n)|\le \lambda/T_i\ \ (1\le i\le n).
\end{equation}
\end{lemma}

\noindent\textit{Proof of the equivalence of\/ {\rm(ii)} and\/ {\rm(iii)} in Lemma~{\rm\ref{dual}}.}
First consider the case when $r_i>0$ for all $i$.
If $n=1$ then there is nothing to prove because (\ref{eee10+}) and (\ref{eee10}) coincide when $c=c'$ and $Q=H$. Thus we will assume that $n\ge2$. Define the linear forms $L_0=u_0$ and $L_i=u_0y_i-u_i$ ($1\le i\le n$). Then the transposed forms are $L_0'=v_0+v_1y_1+\dots+v_ny_n$ and $L'_i=-v_i$ ($1\le i\le n$). It is easily verified that Mahler's lemma is applicable with $d=1$. Let $0<c<1$. Then, the existence of a non-zero integer solution $(q,p_1,\dots,p_n)$ to (\ref{eee10+}) implies the existence of a non-zero integer solution $(u_0,\dots,u_n)$ to (\ref{eee08}) with $T_0=Q$ and $T_i=\delta Q^{-r_i}$ $(1\le i\le n)$, where $\delta=c^{\tau}<1$ and $\tau=\min r_i>0$. By Mahler's lemma, there is a non-zero integer solution $(v_0,\dots,v_n)$ to (\ref{eee09}), where $\lambda=\delta^n$. This implies (\ref{eee10}) with $H=Q$ and $c'=n\delta$. Note that $c'\to0$ as $c\to0$. Thus if there is $c'>0$ such that the only integer solution to (\ref{eee10}) is $a_0=a_1=\dots=a_n=0$, then there must exist a $c>0$ such that the only integer solution to (\ref{eee10+}) is $q=p_1=\dots=p_n=0$. The converse is proved in exactly the same way by swapping the roles of $L_i$ and $L'_i$ and taking $T_0=c'H^{-1}$, $T_i=H^{r_i}$ and $Q=(n+1)H$.

The case when $\rr$ contains a zero is treated by induction. The case $n=1$ meaning $\rr=(r_1)$ with $r_1\neq0$ has already been done. Assume that $n>1$ and our desired statement holds for smaller dimensions. Assume that $\rr$ contains a zero component. Without loss of generality assume that $r_n=0$.
Since $\|x\|^{1/0}=0$, we have that $\max_{1\le i\le n}\|qy_i\|^{1/r_i}=\max_{1\le i\le n-1}\|qy_i\|^{1/r_i}$. Therefore, $\vv y\in\Badr$ if and only if $\vv y'=(y_1,\dots,y_{n-1})\in\Bad(\rr')$, where $\rr'=(r_1,\dots,r_{n-1})$. By induction, this is equivalent to the existence of $c>0$ such that for any $H\ge1$ the only integer solution $(a_0,a_1,\dots,a_{n-1})$ to the system
$$
|a_0+a_1y_1+\dots+a_{n-1}y_{n-1}|< c H^{-1},\qquad
|a_i|< H^{r_i}\quad (1\le i\le n-1)
$$
is $a_0=\dots=a_{n-1}=0$. In turn, the latter statement is equivalent to (iii), since, by $r_n=0$, the inequality $|a_n|< H^{r_n}$ implies that $a_n=0$ whenever $a_n\in\Z$.
\qed

\section{Proof of \eqref{incl}}\label{B}

Recall that \eqref{incl} is the following inclusion
$$
\cB_n\subset\cW_n^*\cap\cB_n^*.
$$
Since for $n=1$ \eqref{incl} becomes trivial, we will assume that $n\ge2$.
Fix any $\xi\in\cB_n$. Define
$$
c_1(\xi,n)\stackrel{\rm def}{=} \inf_{P\in\Z[x],\ 1\le \deg P\le n}H(P)^n|P(\xi)|.
$$
Note that, by Dirichlet's theorem, $c_1(\xi,n)\le 1$ and, by the assumption that $\xi\in\cB_n$, we have that
\begin{equation}\label{vbvb1}
c_1(\xi,n)>0.
\end{equation}
Also note that $\xi$ is not algebraic of degree $\le n$, since otherwise $c_1(\xi,n)=0$.

Assume for a moment that $\xi\not\in\cB_n^*$. Then there exists a sequence $(\alpha_i)_{i\in\N}$ of algebraic numbers of degree $\le n$ such  that $H(\alpha_i)^{n+1}|\xi-\alpha_i|\to0$ as $i\to\infty$. Let $P_i\in\Z[x]$ be the minimal polynomial of $\alpha_i$ over $\Z$. In particular, $P_i(\alpha_i)=0$, $1\le \deg P_i=\deg\alpha_i\le n$ and $H(P_i)=H(\alpha_i)$. Using Taylor's Theorem, we get that
\begin{align*}
H(P_i)^n|P_i(\xi)|&=H(P_i)^n\left|\sum_{j=1}^n\tfrac{1}{j!}P^{(j)}(\alpha_i)(\xi-\alpha_i)^j\right|\ll \\[1ex] &\ll H(P_i)^{n+1}|\xi-\alpha_i|=H(\alpha_i)^{n+1}|\xi-\alpha_i|\to0
\end{align*}
as $i\to\infty$, contrary to \eqref{vbvb1}. Hence, $\xi$ must be in $\cB_n^*$.

\medskip

In order to show that $\xi\in\cW_n^*$ take $\ve_0 = (1+n^2\max\{1,|\xi|^{n}\})^{-n}c_1(\xi,n)$, any integer $Q>1$ and consider the following system of inequalities:
\begin{equation}\label{vb+z}
\left\{\begin{array}{rcl}
|\sum_{i=0}^n a_i\xi^{i}| & < & \ve_0 Q^{-n},\\[0.5ex]
|\sum_{i=1}^n ia_i\xi^{i-1}| & < & \ve_0^{-1}Q,\\[0.5ex]
|a_i| & \le & Q \qquad (2\le i\le n).
\end{array}\right.
\end{equation}
By Minkowski's theorem for convex bodies, there exists a non-zero integer vector $(a_0,\dots,a_n)$ satisfying this system. Define the polynomial $P=a_nx^n+\dots+a_1x+a_0$. Assume for a moment that $|P'(\xi)|\le Q$. Then, using the above system we get that
$$
|a_1|=\left|P'(\xi)-\sum_{i=2}^nia_i\xi^{i-1}\right|\le (1+n^2\max\{1,|\xi|^{n-1}\})Q
$$
and
$$
|a_0|=\left|P(\xi)-\sum_{i=1}^na_i\xi^{i}\right|\le (1+n\max\{1,|\xi|^{n}\})Q\,.
$$
Thus $H(P)\le (1+n^2\max\{1,|\xi|^{n}\})Q$ and we obtain
$$
H(P)^n|P(\xi)|< (1+n^2\max\{1,|\xi|^{n}\})^n\ve_0=c_1(\xi,n).
$$
This contradicts the definition of $c_1(\xi,n)$. Therefore, we must have that $|P'(\xi)|> Q$. By \eqref{vb+z}, we have that $|P(\xi)|<\ve_0 Q^{-n}$. Hence, by Taylor's formula and the fact that $|P(\xi)|<\ve_0 Q^{-n}<\tfrac12Q^{-n}$, the expression
$$
\frac{P(x)}{x-\xi}=\frac{P(\xi)}{x-\xi}+P'(\xi)+\sum_{i=2}^n\tfrac1{i!}P^{(i)}(\xi)(x-\xi)^{i-1}
$$
has the same sign as $P'(\xi)$ for $x=\xi\pm Q^{-n-1}$ provided that $Q$ is sufficiently large. Hence $P(x)$ must have opposite signs at $\xi-Q^{-n-1}$ and $\xi+Q^{-n-1}$. By continuity, this means that there is a root of $P$, say $\alpha$ in the interval $|x-\xi|\le Q^{-n-1}$.
Once again using \eqref{vb+z} we obtain that $H(P)\le c_2 Q$ with $c_2=(1+n^2\max\{1,|\xi|^{n}\})\ve_0^{-1}$. This means that
\begin{equation}\label{vbvbvb}
|\xi-\alpha|\le Q^{-n-1}\ll H(P)^{-n-1}\,.
\end{equation}
Let $P_\alpha$ denote the minimal polynomial of $\alpha$ over $\Z$. Since $P(\alpha)=0$, by Gauss's lemma, $P_\alpha$ divides $P$, that is $P=P_\alpha R$ for some $R\in\Z[x]$. Then, by \cite[\S2, Lemma~8]{Sprindzuk-Mahler}, we have that $H(P)=H(P_\alpha R)\gg H(P_\alpha)H(R)\ge H(P_\alpha)=H(\alpha)$, where the constant in the Vinogradov symbol depends on $n$ only. Thus, $H(P)\gg H(\alpha)$ and \eqref{vbvbvb} implies that
\begin{equation}\label{vbvbvb2}
|\xi-\alpha|\le Q^{-n-1}\ll H(\alpha)^{-n-1}\,.
\end{equation}
Note that if the same $\alpha$ turned up in the above construction for infinitely many $Q$, then $\xi$ would be equal to  this $\alpha$. However, this is impossible, since, as we noted just after \eqref{vbvb1}, $\xi$ cannot be algebraic of degree $\le n$. Therefore, there must be infinitely many different real algebraic numbers $\alpha$ of degree $\le n$ satisfying \eqref{vbvbvb2}. This means that $\xi\in\cW_n^*$. The proof is thus complete.
\qed

\section{Proof of the Fibering Lemma}\label{C}

Here we give a proof of the Fibering Lemma stated in \S\ref{RC} on page~\pageref{SFL}.
We will need the following technical statement.

\begin{lemma}\label{lem12}
  Let $0<d_0<d$ be integers and let $e_d:\Z^m_{\ge0}\to\Z_{\ge0}$ be given by
  \begin{equation}\label{map}
e_d(\alpha_1,\dots,\alpha_m)\stackrel{\rm def}{=} \sum_{j=1}^m\alpha_j(d^{j-1}+d^m)\,.
\end{equation}
Let
\begin{equation}\label{Sd0}
S_{d_0}\stackrel{\rm def}{=} \{(\alpha_1,\dots,\alpha_m)\in\Z_{\ge0}^m: \alpha_1+\dots+\alpha_m\le d_0\}\,.
\end{equation}
Then
\begin{itemize}
  \item[{\rm(i)}] $e_d$ maps $S_{d_0}$ into $\Z_{\ge0}$ injectively, and
  \item[{\rm(ii)}] $e_d(S_{d_0})\cap e_d(\Z_{\ge0}^m\setminus S_{d_0})=\emptyset$.
\end{itemize}
\end{lemma}

\begin{proof}
Let $(\alpha_1,\dots,\alpha_m)$ and $(\alpha'_1,\dots,\alpha'_m)$ be two different elements of $S_{d_0}$ and let $k$ be the largest index such that $\alpha_k\neq\alpha'_k$. Note that
\begin{equation}\label{vbz1}
\left|\sum_{j=1}^{m}(\alpha_j-\alpha'_j)d^{j-1}\right|\le \sum_{j=1}^{m} d_0d^{j-1}=d_0\frac{d^m-1}{d-1}\le d^m-1\,.
\end{equation}
If $\alpha_1+\dots+\alpha_m\neq \alpha'_1+\dots+\alpha'_m$ then
\begin{align*}
|e_d(\alpha_1,\dots,\alpha_m)&-e_d(\alpha_1,\dots,\alpha_m)|  \ge d^m\left|\sum_{j=1}^m\alpha_j-\sum_{j=1}^m\alpha'_j\right|-\\[0ex]
&-\left|\sum_{j=1}^{m}(\alpha_j-\alpha'_j)d^{j-1}\right|~\stackrel{\eqref{vbz1}}{\ge}~ d^m-(d^m-1)=1\,.
\end{align*}
Thus, $e_d(\alpha_1,\dots,\alpha_m)\neq e_d(\alpha'_1,\dots,\alpha'_m)$ in this case.
Now if $\alpha_1+\dots+\alpha_m=\alpha'_1+\dots+\alpha'_m$ then
$$
|e_d(\alpha_1,\dots,\alpha_m)-e_d(\alpha_1,\dots,\alpha_m)|=\left|\sum_{j=1}^{m}(\alpha_j-\alpha'_j)d^{j-1}\right|=
$$
$$=
\left|\sum_{j=1}^{k}(\alpha_j-\alpha'_j)d^{j-1}\right|
\ge d^{k-1}|\alpha_k-\alpha'_k|-\sum_{j=1}^{k-1}|\alpha_j-\alpha_j'|d^{j-1}\ge
$$
$$
\ge d^{k-1}-\sum_{j=1}^{k-1}d_0d^{j-1}=\left\{\begin{array}{cl}
1 & \text{if }k=1\\
d^{k-1}-d_0\frac{d^{k-1}-1}{d-1} & \text{if }k>1
                                              \end{array}
\right.~~\ge 1\,.
$$
Again we obtain that $e_d(\alpha_1,\dots,\alpha_m)\neq e_d(\alpha'_1,\dots,\alpha'_m)$ and thus prove part (i) of the lemma.
Finally, observe that
$$
\max e_d(S_{d_0})=d_0(d^{m-1}+d^m)<\min e_d(\Z_{\ge0}^m\setminus S_{d_0})=(d_0+1)(1+d^m)\,,
$$
whence (ii) readily follows.
\end{proof}

\bigskip

\noindent\textit{Proof of the Fibering Lemma.} Since $f_0,\dots,f_n$ are analytic we can write them as the following absolutely convergent power series
  $$
  f_i(x_1,\dots,x_m)=\sum_{\alpha_1,\dots,\alpha_m\ge0}\lambda^{(i)}_{\alpha_1,\dots,\alpha_m}x_1^{\alpha_1}\cdots x_m^{\alpha_m}\,.
  $$
Since they are linearly independent over $\R$ for every $(c_0,\dots,c_n)\in\R^{n+1}\setminus\{\vv0\}$ the function
$$
\sum_{i=0}^n c_if_i(x_1,\dots,x_m)=
\sum_{\alpha_1,\dots,\alpha_m\ge0}\sum_{i=0}^nc_i\lambda^{(i)}_{\alpha_1,\dots,\alpha_m}x_1^{\alpha_1}\cdots x_m^{\alpha_m}
$$
is not identically zero. Hence, there exist a multiindex $(\alpha_1,\dots,\alpha_m)\in\Z_{\ge0}^m$ such that
\begin{equation}\label{vbf}
\sum_{i=0}^nc_i\lambda^{(i)}_{\alpha_1,\dots,\alpha_m}\neq0\,.
\end{equation}
Therefore, the collection of the sets
$$
\cC(\alpha_1,\dots,\alpha_m) = \big\{(c_0,\dots,c_n)\in\R^{n+1}:\sum_{i=0}^nc_i^2=1,\ \eqref{vbf}\text{ holds}\big\}
$$
taken over $(\alpha_1,\dots,\alpha_m)\in\Z^m_{\ge0}$ is an open cover of the unit sphere in $\R^{n+1}$. Since the sphere is compact, there exists a finite subcover, say, $\cC(\alpha^{(1)}_1, \dots,\alpha^{(1)}_m)$, \dots, $\cC(\alpha^{(N)}_1,\dots,\alpha^{(N)}_m)$.
Let
$$
d_0 = \max\{\alpha^{(\ell)}_1+\cdots+\alpha^{(\ell)}_m:1\le \ell \le N\}\,.
$$
Then, for every non-zero collection $c_0,\dots,c_n$ there exists a multiindex $(\alpha_1,\dots,\alpha_m)\in S_{d_0}$, where $S_{d_0}$ is given by \eqref{Sd0},
such that
$$
\sum_{i=0}^nc_i\lambda^{(i)}_{\alpha_1,\dots,\alpha_m}\neq0\,.
$$
Take any integer $d>d_0$ and any $\vv u=(u_1,u_2,\dots,u_m)\in \R^{m}$ with $u_1\cdots u_m\neq0$. Then, by what we have just shown,
\begin{equation}\label{vm}
\sum_{i=0}^nc_i\lambda^{(i)}_{\alpha_1,\dots,\alpha_m}\prod_{j=1}^mu_j^{\alpha_j}\neq0\qquad\text{for some }(\alpha_1,\dots,\alpha_m)\in S_{d_0}\,.
\end{equation}
Note that\\[-5ex]
\begin{align}
\nonumber\phi_{\vv u,i}(t) & =
\sum_{\alpha_1,\dots,\alpha_m\ge0}\lambda^{(i)}_{\alpha_1,\dots,\alpha_m}\prod_{j=1}^m(u_jt^{d^{j-1}+d^m})^{\alpha_j}\\[0ex]
&=\sum_{\alpha_1,\dots,\alpha_m\ge0}\lambda^{(i)}_{\alpha_1,\dots,\alpha_m}
\prod_{j=1}^mu_j^{\alpha_j}t^{e_d(\alpha_1,\dots,\alpha_m)}\,,\label{vbf2}
\end{align}
where $e_d$ is given by \eqref{map}.
Consider the linear the combination of functions \eqref{vbf2} with coefficients $c_0,\dots,c_n$:
$$
\sum_{i=0}^nc_i\phi_{\vv u,i}(t)=
\sum_{i=0}^nc_i\sum_{\alpha_1,\dots,\alpha_m\ge0}\lambda^{(i)}_{\alpha_1,\dots,\alpha_m}
\prod_{j=1}^mu_j^{\alpha_j}t^{e_d(\alpha_1,\dots,\alpha_m)}=
$$
$$
\,\qquad\qquad=\sum_{\alpha_1,\dots,\alpha_m\ge0}\sum_{i=0}^nc_i\lambda^{(i)}_{\alpha_1,\dots,\alpha_m}
\prod_{j=1}^mu_j^{\alpha_j}t^{e_d(\alpha_1,\dots,\alpha_m)}\,.
$$
By Lemma~\ref{lem12} and \eqref{vm}, the above series in $t$ is not identically zero.
Since $(c_0,\dots,c_n)\neq0$ is arbitrary, the functions \eqref{vbf2} are linearly independent over $\R$.\\[-1ex]
\hspace*{\fill}$\boxtimes$

\bigskip
\bigskip

\noindent\textit{Acknowledgements.} The author is grateful to Maurice Dodson, Sanju Velani and Dmitry Kleinbock for their valuable comments on an earlier version of this paper and to Evgeniy Zorin whose suggestion of the modification of Sprind\v zuk's fibering technique (specifically Lemma~\ref{lem12}) helped filling a gap in Sprind\v zuk's argument. The author is also really grateful to the anonymous reviewer of this paper for the very detailed report providing many helpful suggestions.

{\parskip=0ex \small

\def\cprime{$'$}
  \def\polhk#1{\setbox0=\hbox{#1}{\ooalign{\hidewidth
  \lower1.5ex\hbox{`}\hidewidth\crcr\unhbox0}}}
\providecommand{\bysame}{\leavevmode\hbox to3em{\hrulefill}\thinspace}

}

\end{document}